\documentclass[10pt]{amsart}
\usepackage{egastyle}

\usepackage{amssymb,amsmath,amscd,amsthm,amsfonts,wasysym,mathrsfs,amsxtra,stmaryrd,array,mathtools,xcolor,soul}
\usepackage{appendix}

\usepackage{graphicx,array,url}
\usepackage[vcentermath]{genyoungtabtikz}
\usepackage{tikz}
\usepackage[section]{placeins}
\usepackage{afterpage}
\usepackage{verbatim}

\usepackage{tikz-cd}

\input xy
\xyoption{all}

\usepackage[cmtip,arrow]{xy}
   \usepackage{pb-diagram,pb-xy}

\newcommand{\RR}{\mathbb{R}}
\newcommand{\OO}{\mathcal{O}}
\newcommand{\ZZ}{\mathbb{Z}}

\newcommand{\Hom}{\textnormal{Hom}}

\newcommand{\Ac}{\mathcal{A}}
\newcommand{\Bc}{\mathcal{B}}

\newcommand{\Fc}{\mathcal{F}}
\newcommand{\Tc}{\mathcal{T}}

\newcommand{\Hc}{\mathcal{H}}

\newcommand{\Pc}{\mathcal{P}}

\newcommand{\Coh}{\mathrm{Coh}}

\newcommand{\arinj}{\ar@{^{(}->}}
\newcommand{\arsurj}{\ar@{->>}}
\newcommand{\areq}{\ar@{=}}

\newcommand{\wh}{\widehat}

\newcommand{\ch}{\mathrm{ch}}

\newcommand{\Pic}{\mathrm{Pic}}

\newcommand{\whPhi}{{\wh{\Phi}}}

\newcommand{\wt}{\widetilde}

\newcommand{\bsm}{\begin{smallmatrix}}
\newcommand{\esm}{\end{smallmatrix}}

\newcommand{\Aut}{\mathrm{Aut}}
\newcommand{\Stab}{\mathrm{Stab}}
\newcommand{\Hilb}{\mathrm{Hilb}}

\makeatletter
\newtheorem*{rep@theorem}{\rep@title}
\newcommand{\newreptheorem}[2]{%
\newenvironment{rep#1}[1]{%
 \def\rep@title{#2 \ref{##1}}%
 \begin{rep@theorem}}%
 {\end{rep@theorem}}}
\makeatother

\newreptheorem{theorem}{Theorem}

\usepackage{scalerel,stackengine}
\stackMath
\newcommand\reallywidehat[1]{%
\savestack{\tmpbox}{\stretchto{%
  \scaleto{%
    \scalerel*[\widthof{\ensuremath{#1}}]{\kern-.6pt\bigwedge\kern-.6pt}%
    {\rule[-\textheight/2]{1ex}{\textheight}}
  }{\textheight}%
}{0.5ex}}%
\stackon[1pt]{#1}{\tmpbox}%
}
\begin{document}

\title[Stability conditions under autoequivalences]{Geometric stability conditions under autoequivalences and applications: Elliptic Surfaces}

\author[Jason Lo]{Jason Lo}
\address{Department of Mathematics \\
California State University, Northridge\\
18111 Nordhoff Street\\
Northridge CA 91330 \\
USA}
\email{jason.lo@csun.edu}

\author[Cristian Martinez]{Cristian Martinez}
\address{Institute of Mathematics, Statistics and Scientific Computing \\
Universidade Estadual de Campinas\\
Rua Sergio Buarque de Holanda, 651\\
13083-859, Campinas, SP \\
Brazil}
\email{cristian@unicamp.br}

\keywords{Elliptic surfaces, Fourier-Mukai transforms, stability conditions}
\subjclass[2020]{Primary 14D20, 14J27; Secondary: 14F08}

\begin{abstract}
On a Weierstra\ss\ elliptic surface, we describe the action of the relative Fourier-Mukai transform on the geometric chamber of $\Stab(X)$, and in the K3 case we also study the action on one of its boundary compenents. Using new estimates for the Gieseker chamber we prove that Gieseker stability for polarizations on certain Friedman chamber is preserved by the derived dual of the relative Fourier-Mukai transform. As an application of our description of the action, we also prove projectivity for some moduli spaces of Bridgeland semistable objects.
\end{abstract}

\maketitle
\tableofcontents


\section{Introduction}
Ever since the seminal work of Mukai \cite{mukai_1981} on duality for abelian varieties, researchers have been wondering what type of stability is satisfied by the image of a stable sheaf under an autoequivalence of the derived category. In fact, Mukai's ideas inspired several works and generalizations (\cite{FMTes,BMef,YosPI} just to mention a few). One of the main obstacles for the analysis of the stability of an image object relies on the fact that, a priori, the image of a sheaf is not a sheaf. This forced the first authors working on this question to study the stability of the cohomologies of the image objects instead, and to give conditions under which a sheaf was taken to (a shift of) a sheaf by the autoequivalence. This setting changed a couple of decades ago with Bridgeland's introduction of stability conditions on triangulated categories \cite{StabTC} and his family of examples on K3 surfaces \cite{SCK3}. Then, it was clear that the image of a stable sheaf on a surface by any autoequivalence was a complex satisfying certain notion of stability in some abelian subcategory of the derived category.

\subsection{Solving an equation of stability} Given a triangulated category $\Tc$, one can associate the space $\mathrm{Stab}(\Tc)$ of Bridgeland stability conditions which comes with  two commuting group actions: a left action by the autoequivalence group $\mathrm{Aut}(\Tc)$ of the triangulated category, and a right action by $\widetilde{\mathrm{GL}}^+\!(2,\mathbb{R})$, the universal cover of $\mathrm{GL}^+(2,\mathbb{R})$.  In this article, we give a class of solutions to the following general problem:

\begin{prob}\label{prob:solvestabeq}
Given an autoequivalence $\Phi$ of a triangulated category $\Tc$, and a Bridgeland stability condition $\sigma$ on $\Tc$, can we find a Bridgeland stability condition $\sigma'$ and an element $g \in \widetilde{\mathrm{GL}}^+\!(2,\mathbb{R})$ so that the equation
\begin{equation}\label{eq:stabeq0}
  \Phi \cdot \sigma = \sigma' \cdot g
\end{equation}
holds?
\end{prob}

The point here is to construct $\sigma'$ directly (without referencing $\Phi$), so that we have two stability conditions $\sigma, \sigma'$ that can be constructed independently, and such that they satisfy the relation \eqref{eq:stabeq0}.  Since the $\widetilde{\mathrm{GL}}^+\!(2,\mathbb{R})$-action on moduli spaces does not alter semistable objects and merely relabels their phases, having an equation such as  \eqref{eq:stabeq0} means we ``completely understand'' the image of $\sigma$ under $\Phi$.

Suppose $\Tc = D^b(X)$ is the bounded derived category of coherent sheaves on a smooth projective variety $X$, and $\Phi^{\ch}$ denotes the endomorphism on the Chow ring of $X$ induced by $\Phi$.  If  we write $\mathcal{M}(\sigma, v)$ to denote the moduli space of $\sigma$-semistable objects in $D^b(X)$ with Chern character $v$, then an equation of stability conditions of the form \eqref{eq:stabeq0} implies, for any $v$, that we have a set-theoretic bijection between the closed points of the  moduli spaces
\[
  \mathcal{M}_{\sigma}(v) \leftrightarrow \mathcal{M}_{\Phi \cdot \sigma}( \Phi^{\ch}(v))
\]
simply because $\Phi$ is an autoequivalence of categories.  In case $\Phi$ is an equivalence that respects flat families of objects, such as a relative Fourier-Mukai transform on elliptic fibrations, we obtain an isomorphism of moduli spaces
\[
  \mathcal{M}_{\sigma}(v) \cong \mathcal{M}_{\Phi \cdot \sigma}( \Phi^{\ch}(v)).
\]
That is, anytime we have an equation of stability conditions of the form \eqref{eq:stabeq0}, we obtain an isomorphism of moduli spaces for \emph{any} choice of Chern character $v$.  This provides a method for studying moduli spaces on $X$ that takes a different viewpoint from the more traditional approach, where one begins with a \emph{fixed} Chern character $v$, thus fixing some moduli space $\mathcal{M}'(v)$ with respect to some notion of stability, and then ask whether there are transformations one can perform on $\mathcal{M}'(v)$ (such as using autoequivalences on $D^b(X)$) to bring it to a more familiar moduli space.

\subsection{Relations to other problems}  Problem \ref{prob:solvestabeq} is a fundamental problem  in contexts ranging from mirror symmetry, Donaldson-Thomas invariants, to dynamical systems.  We point out some of them to highlight its relevance:


\textbf{K3 and elliptic surfaces.} Mathematically, the stability manifold $\mathrm{Stab}(X)$ of a smooth projective variety $X$ and the various structures that can be derived from it can be considered as  invariants of $X$.  In the case of a K3 surface $X$, Bridgeland conjectures that the action of $\mathrm{Aut}(D^b(X))$ preserves the connected component $\mathrm{Stab}^\dagger (X)$ of $\mathrm{Stab}(X)$ consisting of `geometric' stability conditions \cite[Conjecture 1.2]{SCK3}.  In the case of K3 surfaces of Picard rank 1, this conjecture was proved by Bayer and Bridgeland \cite{bayer2017derived} by showing that the union of images of $\mathrm{Stab}^\dagger (X)$ under the $\mathrm{Aut}(D^b(X))$-action is contractible.  Results such as \eqref{eq:stabeq0} can also help towards this conjecture: When $X$ is a Weierstra{\ss} elliptic surface of nonzero Kodaira dimension, the autoequivalence group has an explicit description by the work of Uehara \cite{uehara2015autoequivalences}, and the analogue of Bridgeland's conjecture can be verified directly using equations of the form \eqref{eq:stabeq0} \cite[Theorem 13.2]{Lo20}.

\textbf{Donaldson-Thomas invariants.}  Motivated by trying to explain symmetries within and constraints on Donaldson-Thomas invariants, and attempting to construct stability conditions on the category of matrix factorisations, Toda proposed to study `Gepner-type' Bridgeland stability conditions, which are stability conditions $\sigma$ on a triangulated category $\Tc$ satisfying
\begin{equation}\label{eq:stabeq2}
  \Phi \cdot \sigma = \sigma \cdot s
\end{equation}
for some $\Phi \in \mathrm{Stab}(\Tc)$ and some complex number $s$ \cite{toda2013gepner}.  Here,  the $\mathbb{C}$-action on $\mathrm{Stab}(\Tc)$ is more restrictive than the $\widetilde{\mathrm{GL}}^+\!(2,\mathbb{R})$-action, and so equation \eqref{eq:stabeq2} is a special case of equation \eqref{eq:stabeq0}.

\textbf{Dynamical systems.}  Recently, a parallel between Teichm\"{u}ller theory of surfaces and Bridgeland stability conditions has been established.  Since a generic element of the mapping class group of a closed orientable surface is a pseudo-Anosov map, it is  natural to ask what the categorical analogue of a pseudo-Anosov map is.  To capture the idea that a pseudo-Anosov map on a Riemann surface streches a foliation while contracting another, Dimitrov, Haiden, Katzarkov and Kontsevich proposed a definition of  pseudo-Anosov equivalences of a triangulated category: An autoequivalence $\Phi$ of a triangulated category $\Tc$ is called pseudo-Anosov if it satisfies an equation of the form
\begin{equation}\label{eq:stabeq3}
  \Phi \cdot \sigma = \sigma \cdot g
\end{equation}
for some $\sigma \in\mathrm{Stab}^\dagger (X)$ and some $g \in \widetilde{\mathrm{GL}}^+\!(2,\mathbb{R})$ where $g$ is a lift of an element $\begin{pmatrix} r & 0 \\ 0 & 1/r \end{pmatrix} \in \mathrm{GL}^+\!(2,\mathbb{R})$ with $|r| \neq 1$ \cite[Definition 2.14]{FAN2021107732}.  Here, the Bridgeland stability condition $\sigma$ plays the role of  measured foliations, while $r, 1/r$ correspond to the factors of stretch/contraction induced by a pseudo-Anosov map. Clearly, equation \eqref{eq:stabeq3} is a special case of \eqref{eq:stabeq0}.

\subsection{Main results and applications} In this article, we give a class of solutions to equation \eqref{eq:stabeq0} on elliptic surfaces.  More concretely, suppose $\pi : X \to Y$ is    a Weierstra{\ss} elliptic surface, and $\Phi \in \mathrm{Aut}(D^b(X))$ is the relative Fourier-Mukai transform with kernel given by 
\begin{equation}\label{kernel}
 \mathcal{P}=\mathcal{I}_{\Delta}\otimes {\pi_1}^*(\mathcal{O}(\Theta))\otimes {\pi_2}^*(\mathcal{O}(\Theta))\otimes {\rho^*\omega}^{-1}
\end{equation}
where $\pi_1,\pi_2\colon X\times_Y X\rightarrow X$ are the projections, $\rho=\pi\circ \pi_1=\pi\circ\pi_2$, and $\omega=R^1\pi_*\mathcal{O}_X$.   That is, the kernel $\mathcal{P}$ is a normalised Poincar\'{e} sheaf that parametrises rank-one, torsion-free sheaves of degree zero on the fibration $\pi$. Let $\Theta$ denote the canonical section of $\pi$, and $f$ the class of a fiber for the fibration. 

\begin{thm}\label{thm:intromain1}
Let $X$ be a Weierstra{\ss} elliptic surface.   For any $\mathbb{R}$-divisor $\bar{B}$ in the span of $\Theta, f$,  there exists $a_0 >0$ such that, for any ample $\mathbb{R}$-divisor  $\bar{\omega} = \Theta + \bar{b}f$ and any $\bar{a}>a_0$, we can find $\mathbb{R}$-divisors $\omega, B$, a constant $a$, and some $g \in \widetilde{\mathrm{GL}}^+\!(2,\mathbb{R})$ such that the following equation holds:
\begin{equation}\label{eq:stabeq5}
  \Phi \cdot \sigma_{\bar{a}, \bar{B},\bar{\omega}} = \sigma_{a, B,\omega} \cdot g.
\end{equation}
\end{thm}
The key point here is that $a, B, \omega, g$ can  be computed explicitly in terms of $\bar{\omega},\bar{B}$, and $\bar{a}$.  That is, given the Bridgeland stability condition $\sigma_{\bar{a}, \bar{B}, \bar{\omega}}$, we know \emph{exactly which} stability condition $\Phi \cdot \sigma_{\bar{a}, \bar{B}, \bar{\omega}}$ is - up to the action of the element $g$, it is precisely $\sigma_{a, B, \omega}$, and we can explicitly write down $g, a, B, \omega$.  In other words, we understand exactly how the autoequivalence $\Phi$ acts on the subset of the stability manifold $\Stab (X)$ consisting of stability conditions of the form $\sigma_{\bar{a}, \bar{B}, \bar{\omega}}$.

In the case where $\bar{\omega}$ is assumed to be an $\mathbb{R}$-multiple of an integral ample divisor and $\bar{B}$ is assumed to be an integral multiple of $f$, equation \eqref{eq:stabeq5} was established in \cite[Theorem 11.7]{Lo20}. In this article, the integrality assumptons are replaced with  an estimate for the Bridgeland stability of  1-dimensional Gieseker stable sheaves on smooth projective surfaces (Lemma \ref{one dim real estimates}).

Armed with Theorem \ref{thm:intromain1}, we explore some consequences related to preservation of stability by $\Phi$, obtaining rational maps, birational maps, and in some cases isomorphisms among different moduli spaces of sheaves. Preservation of slope stability by the autoequivalence $\Phi$, for instance, has been studied under some numerical constrains by several authors \cite{FMTes,yoshioka_1999,Ruiprez2002StableSO,Yoshioka2000ModuliSO,Jardim2000AFA} for polarizations in the so called Friedman chamber, which we describe below:

Given an elliptic surface $\pi\colon X\rightarrow Y$ with generic fiber $f$, a Chern character $v$ with coprime rank and fiber degree, and a polarization $L$, Friedman \cite{FriedmanRank2} proved that there exists $b_0>0$ such that the set of slope semistable torsion-free sheaves with Chern character $v$ and with respect to polarizations of the form $L+bf$ is constant for $b>b_0$. Moreover, stability and semistability are equivalent for all such values of $b$. In the particular case of a Weierstra\ss\ elliptic surface, we can study stability for polarizations of the form $\Theta+bf$ for $b\gg 0$. The corresponding moduli space will be denoted by $M_f(v)$.

The novelty of our approach comes after interpreting slope stability in the realm of Bridgeland stability conditions. Then, instead of analyzing the cohomologies of the complex $\Phi(E)[1]$ when $E$ is a slope stable sheaf, we look for the chamber of the stability manifold $\Stab(X)$ that corresponds to the Gieseker chamber by the autoequivalence $\Phi[1]$. Our analysis can be described as follows:
\begin{itemize}
\item[\textbf{Step 1}.] We start by fixing a $\mathbb{Q}$-divisor $\bar{B}$ and a polarization $\bar{\omega}$, and consider certain path $\{\sigma_{\bar{a},\bar{B},\bar{\omega}}\}_{\bar{a}>0}$ of stability conditions in $\Stab(X)$. Then, for a Chern character $v$, we give precise estimates for $\bar{a}$ so that this path enters the Gieseker chamber, i.e., for which values of $\bar{a}$ one has that the only $\sigma_{\bar{a},\bar{B},\bar{\omega}}$-semistable objects of Chern character $v$ are precisely the ($\bar{B}$-twisted) Gieseker semistable sheaves with respect to $\bar{\omega}$. 
\item[\textbf{Step 2}.] The next step consists in proving that when $\bar{B}=\bar{p}\Theta+\bar{q}f$ and $\bar{\omega}=\Theta+\bar{b}f$ then $\Phi[1]\cdot \sigma_{\bar{a},\bar{B},\bar{\omega}}$-stability is equivalent to $\sigma_{a,B,\omega}$-stability for some $B=p\Theta+qf$, $\omega=\Theta+bf$, and $a>0$. In particular this holds for large values of $\bar{a}$, when $\sigma_{\bar{a},\bar{B},\bar{\omega}}$ lies in the Gieseker chamber.
\item[\textbf{Step 3}.] Ideally, one would like to impose conditions on $\bar{B},\bar{\omega}$ and $\bar{a}$ to ensure that the path $\{\sigma_{a,B,\omega}\}_{a>0}$ enters the Gieseker chamber for $\Phi(v)[1]$ at the same time that the path $\{\sigma_{\bar{a},\bar{B},\bar{\omega}}\}_{\bar{a}>0}$ enters the Gieseker chamber for $v$, providing the perfect set up to compare the Gieseker stability of a sheaf and its relative Fourier-Mukai transform. However, this can be only be achieved for particular values of $v$, two of which we study in detail: when the fiber degree is zero and $\bar{B}$ is a multiple of the fiber class (Theorem \ref{zerofiberdegree}), and when $v$ is the Chern character of a line bundle and the Picard number is two (Corollary \ref{linebundles}); in both cases we prove that $\Phi(E)$ is the derived dual of a Gieseker semistable sheaf.
\item[\textbf{Step 4}.] In the particular case when $\bar{B}$ is a multiple of the fiber, the ray $\{\sigma_{\bar{a},\bar{B},\bar{\omega}}\}_{\bar{a}>0}$ enters the Gieseker chamber for $v$ at the same time that $\omega$ enters the Friedman chamber for $\Phi(v)^{\vee}$ and vice-versa. However, since the Gieseker chambers do not exactly correspond to each other before and after the relative Fourier-Mukai transform, the autoequivalence $\Phi^{\vee}$ induces only a rational map, obtained from the well-behaved Bridgeland wall-crossing, between the  corresponding Gieseker moduli spaces. We can then use a similar division algorithm to the one  proposed by Bernardara and Hein in \cite{bernardara:hal-00958218} in the case of elliptic K3 surfaces to show that if the rank and fiber degree are coprime then any twisted Gieseker moduli space is birational to a moduli space of rank 1 torsion-free sheaves. 
\end{itemize}
Our computations and intuition of what the image of the stability condition $\sigma_{\bar{a},\bar{B},\bar{\omega}}$ by the autoequivalence $\Phi[1]$ was, have been greatly developed and inspired by earlier work of the authors together with Liu \cite{LLM}, where the image of slope stability for torsion-free sheaves was identified with a limit stability condition. There, the limit stability of the image of 1-dimensional Gieseker semistable sheaves was also studied by a rigurous analysis of the walls for Bridgeland stability conditions. 

Finally, we would like to remark that our approach avoids completely having to verify if a sheaf is $\Phi$-$\text{WIT}_i$. Instead we use the derived dual functor to analyze the stability of the complex $\Phi(E)[1]$ as a whole. 

The organization of the paper is as follows. In Section 2, we recall some basic definitions regarding Weierstra\ss\ elliptic surfaces and stability conditions. In Section 3, we introduce the type of stability conditions we will consider and prove Proposition \ref{isoUpsilon}, which establishes an isomorphism between various Bridgeland moduli spaces induced by an autoequivalence that preserves geometricity. In Section 4, we give specific bounds for the Gieseker chamber. In Section 5, we restrict to the case of Weierstra\ss\ elliptic surfaces and prove Theorem \ref{thm:intromain1}; give a class of examples of projective Bridgeland moduli spaces; and prove Theorem \ref{newisobyPhi}, Corollary \ref{linebundles} and Theorem \ref{zerofiberdegree}, which are our main results about preservation of stability. In Section 6, we discuss the division algorithm. Finally, in Section 7, for the case of Weiertra\ss\ elliptic K3 surfaces we analyze the image by $\Phi$ of some of the stability conditions on the boundary of the geometric chamber described by Tramel and Xia in \cite{tramel2017bridgeland}, and establish Corollary \ref{projectivityTXcomponent}, proving that such stability conditions have projective moduli.

\subsection*{Acknowledgments} The last part of this work was completed while the authors attended the AIM SQuaRE Research Program ``Moduli of sheaves on surfaces via Bridgeland stability'' in  August 2022. The authors would like to thank the American Institute of Mathematics for its support and excellent working conditions. The authors are also grateful to Marcos Jardim, Wei-Ping Li, Zhenbo Qin, and Ziyu Zhang for several insightful discussions on the topic of this article. CM would like to express his gratitude to California State University, Northridge for hosting him during several stages of this project. JL is partially supported by NSF grant DMS-2100906. CM is supported by the FAPESP grant number 2020/06938-4, which is part of the FAPESP Thematic Project 2018/21391-1. 

\section{Preliminaries}

\paragraph For any smooth projective variety $X$, we will always write $D^b(X)$ to denote $D^b(\Coh (X))$, the bounded derived category of coherent sheaves on $X$. 

\paragraph[Weierstra{\ss} elliptic surface] By a Weierstra{\ss} elliptic surface, we mean a flat morphism $\pi : X \to Y$ between smooth projective varieties $X, Y$  over $\mathbb{C}$ such that 
\begin{itemize}
    \item The fibers of $\pi$ are Gorenstein curves of arithmetic genus 1 and geometrically integral.
    \item There exists a section $s : Y \to X$ such that its image $\Theta$ does not intersect any singular point of any singular fiber.
\end{itemize}
Under these conditions, the generic fiber of $\pi$ is a smooth elliptic curve, and the singular fibers are either nodal or cuspidal.  The definition we use here follows that in \cite[Definition 6.10]{FMNT}.  For brevity, we sometimes refer to $X$ as the Weierstra\ss\  elliptic surface. 

We usually write $f$ to denote the divisor class of a fiber of $\pi$ in $\mathrm{NS}(X)$, and we set $e = -\Theta^2$.

\paragraph[Autoequivalences] Let $\pi : X \to Y$ be a Weierstra{\ss} elliptic surface.  Then we always have a pair of relative Fourier-Mukai transforms   $\Phi, \whPhi : D^b(X) \to D^b(X)$ satisfying
\[
  \whPhi \Phi [1] \cong \mathrm{id}_{D^b(X)} \cong \Phi \whPhi [1].
 \]
The autoequivalence $\Phi$ has the normalized relative Poincar\'{e} sheaf \eqref{kernel} as its kernel, and for every skyscraper sheaf $\OO_x$ supported at a closed point $x \in X$, the transform $\Phi (\OO_x)$ is a rank-one, torsion-free sheaf on the fiber of $\pi$ containing $x$.  The functor $\whPhi$ has a similar interpretation.  Details of the construction of the functors $\Phi$ and $\whPhi$ can be found in \cite[6.2.3]{FMNT}.

The automorphisms of the lattice  $\widetilde{\mathrm{NS}}(X):=\mathbb{Z}\oplus\mathrm{NS}(X)\oplus \frac{1}{2}\mathbb{Z}$,  induced by $\Phi$ and $\whPhi$ via the Chern character, can be described as follows \cite[6.2.6]{FMNT}:  For any $E \in D^b(X)$, if we write
\begin{equation}\label{eq:chE}
  \ch_0(E)=n, \text{\quad} f\ch_1(E)=d, \text{\quad}\Theta \ch_1(E)=c, \text{\quad} \ch_2(E)=s,
\end{equation}
then 
\begin{align*}
    \ch_0 (\Phi E) &= d \\
    \ch_1 (\Phi E) &\equiv -c_1(E) + (d-n)\Theta +(c+\tfrac{e}{2}d+s)f \\
    \ch_2(\Phi E) &= -(c+ed-\tfrac{e}{2}n).
 \end{align*}
In particular, we have
\[
f\ch_1(\Phi E)=-n, \text{\quad} \Theta \ch_1(\Phi E)=s-\tfrac{e}{2}d+en.
\]
Also, 
\begin{align*}
    \ch_0 (\whPhi E) &= d \\
    \ch_1 (\whPhi E) &\equiv \ch_1(E)  -(d+n)\Theta + (-c-\tfrac{e}{2}d+s)f \\
    \ch_2 (\whPhi E) &= -(c+ed + \tfrac{e}{2}n).
\end{align*}

\paragraph[Slope stability] Given a smooth projective surface $X$, for any $\mathbb{R}$-divisor $B$ on $X$ and any coherent sheaf $E$ on $X$, we define the $B$-twisted Chern character vector by 
\[
  \ch^B(E) := e^{-B}\ch(E)=(\ch_0^B(E),\ch_1^B(E),\ch_2^B(E))\in \widetilde{\mathrm{NS}}(X)\otimes \mathbb{R}. 
\]
More precisely, we have 
\begin{align*}
    \ch_0^B(E) &= \ch_0(E), \\
    \ch_1^B(E) &= \ch_1(E)-B\ch_0(E), \\
    \ch_2^B(E) &= \ch_2(E)-B\ch_1(E)+\frac{B^2}{2}\ch_0(E).
\end{align*}

Additionally, if $\omega$ is an ample class, then we can define a slope function $\mu_{B,\omega}$ on nonzero coherent sheaves via
\[
  \mu_{B,\omega} (E) = \begin{cases} \frac{\omega\ch_1^B(E)}{\ch_0(E)} &\text{ if } \ch_0(E) \neq 0, \\
  +\infty &\text{ otherwise.}
  \end{cases}
\]

A nonzero coherent sheaf $E$ on $X$ is then called $\mu_{B,\omega}$-semistable (resp.\ $\mu_{B,\omega}$-stable) if
\[
  \mu_{B,\omega}(F) \leq (\text{resp.\ $<$})\,\, \mu_{B,\omega}(E)
\]
for every nonzero proper subsheaf $F\hookrightarrow E$. In particular, a $\mu_{B,\omega}$-semistable sheaf of nonzero rank is torsion-free.  When $B=0$, we usually drop the subscript $B$ in $\mu_{B,\omega}$ and simply write $\mu_\omega$.

When $E$ is a sheaf of nonzero rank, we have $\mu_{B,\omega}(E)=\mu_\omega (E)-B\omega$, meaning $\mu_{B,\omega}$-semistability is equivalent to $\mu_\omega$-stability for $E$.

An important property of $\mu_{B,\omega}$-semistable sheaves is that they satisfy the Bogomolov inequality 
$$
\Delta(E):=\ch_1(E)^2-2\ch_0(E)\ch_2(E)\geq 0.
$$
\paragraph[Bridgeland stability conditions] For \label{para:Bristabconintro} the purpose of this article, a Bridgeland stability condition on $D^b(X)$ is a pair $\sigma = (Z,\Ac)$ where:
\begin{itemize}
\item $\Ac$ is the heart of a bounded t-structure on $D^b(X)$.
\item $Z : K(\Ac) \to \mathbb{C}$ is a group homomorphism from the Grothendieck group of $D^b(X)$  to the additive group of complex numbers, which factors through a surjective map 
$$
v\colon K(\Ac)\twoheadrightarrow \Gamma.
$$
to a finite rank lattice $\Gamma$. Here we are using the identification of Grothendieck groups $K(D^b(X))=K(\Ac)$.  We refer to $Z$ as the central charge of $\sigma$.
\item $Z$ and $\Ac$ satisfy the following conditions:
\begin{itemize}
    \item[(i)](Positivity) For every nonzero object $E$ in $\Ac$ we can write $Z(E)=m(E)\exp(\pi i\phi(E))$ for some $m(E)>0$ and $\phi(E)\in (0,1]$. We refer to $\phi(E)$ as the phase of $E$. We say that $E\in \Ac$ is $Z$-semistable (resp. $Z$-stable) if for any nonzero  proper subobject $A\hookrightarrow E$ in $\Ac$ we have
    $$
    \phi(A)\leq\ (\text{resp.}\ <)\  \phi(E).
    $$
    \item[(ii)](HN filtration) The function $Z$ satisfies the Harder-Narasimhan (HN) property on $\Ac$, i.e., for any nonzero object $E \in \Ac$ there exists a unique filtration in $\Ac$
\[
0 = E_0 \subseteq E_1 \subseteq \cdots \subseteq E_m=E,
\]
with factors $F_i:=E_i/E_{i-1}$ that are $Z$-semistable of strictly decreasing phases:
\[
 \phi (F_1) > \phi (F_2) > \cdots > \phi (F_m).
\]
    \item[(iii)](Support property) There exists a quadratic form $Q\colon \Gamma\otimes \mathbb{R}\rightarrow \mathbb{R}$ such that $\ker Z\subset \Gamma\otimes \mathbb{R}$ is negative definite with respect to $Q$, and $Q(v(E))\geq 0$ for any $Z$-semistable object $E\in D^b(X)$ (see Remark \ref{remRealphase} below). 
\end{itemize}
\end{itemize}
\begin{rem}\label{remRealphase}
Given a stability condition $\sigma = (Z,\Ac)$ on $D^b(X)$, we say an object $E \in D^b(X)$ is $\sigma$-semistable (resp.\ $\sigma$-stable) if it is of the form $E'[j]$ for some $Z$-semistable (resp.\ $Z$-stable) object $E'$ in $\Ac$ and some integer $j$.  If an object $E \in D^b(X)$ is of the form $E = F[m]$ for some $m \in \Ac$, then we define the phase of $E$ via $\phi (E)=\phi (F)+m$.
\end{rem}
\begin{rem}\label{slicing}
For any $\phi\in\mathbb{R}$, we denote by $\mathcal{P}(\phi)\subset D^b(X)$ the full subcategory generated via extension closure by $\sigma$-semistable objects of phase $\phi$. Notice that by the HN property $$
\mathcal{P}(0,1]:=\langle E\colon E\in\mathcal{P}(\phi),\ \phi\in(0,1]\rangle=\mathcal{A}.
$$
The collection of subcategories $\mathcal{P}:=\{\mathcal{P}(\phi)\}_{\phi\in\RR}$ form what is called a slicing of $D^b(X)$.
\end{rem}

\paragraph[Group actions] The set $\Stab^{\Gamma}(X)$ of stability conditions (as defined above) on a smooth projective variety $X$ has the structure of complex manifold \cite{StabTC}, whose local charts are given by the forgetful maps $\Stab^{\Gamma}(X)\rightarrow \Hom(\Gamma,\mathbb{C})$,  which send a stability condition to its central charge. 

The stability manifold comes equipped with two natural commuting group actions \cite[Lemma 8.2]{StabTC}:

\begin{itemize}
    \item A left action of the group of exact autoequivalences $\Aut(D^b(X))$ defined by
    $$
    \Upsilon\cdot (Z,\Ac)=(Z',\Upsilon(\Ac)),\ \text{where}\ Z'(E)=Z(\Upsilon^{-1}E).
    $$
    \item A right action of the universal cover $\widetilde{\mathrm{GL}}^+(2,\RR)$ of $\mathrm{GL}^+(2,\RR)$, the group of $2\times 2$ matrices with positive determinant. Using the description
    \begin{multline*}
    \widetilde{\mathrm{GL}}^+(2,\RR)= \\
    \{(T,f)\colon T\in \mathrm{GL}^+(2,\RR),\ f\colon \RR\rightarrow \RR\ \text{increasing},\ f(\phi+1)=f(\phi)+1,\ f|_{\RR/2\ZZ}=T|_{(\RR^2\setminus\{0\})/\RR_{>0}}\},
    \end{multline*}
    this action is given by:
    $$
     (Z,\mathcal{P})\cdot (T,f)=(Z',\mathcal{P}'),\text{where}\ Z'(E)=T^{-1}Z(E),\ \mathcal{P}'(\phi)=\mathcal{P}(f(\phi)).
    $$
    Here we are defining the action using  slicings, which as explained in Remark \ref{slicing} determines the heart of the corresponding t-structure. 
\end{itemize}

\section{Stability conditions on surfaces} 
Suppose $X$ is a smooth projective surface, and $\omega, B$ are $\mathbb{R}$-divisors on $X$ where $\omega$ is ample.  Then we can define a Bridgeland stability condition $\sigma_{B,\omega}\in \Stab^{\Gamma}(X)$, where $\Gamma=\widetilde{\mathrm{NS}}(X)$ and $v$ is the Chern character map (where $\Gamma, v$ are as in \ref{para:Bristabconintro}), as follows.  First, we define the following full subcategories of $\Coh (X)$:
\begin{align*}
  \Tc_{B,\omega} &= \{ E \in \Coh (X)\colon \mu_{B,\omega}(Q)>0\ \text{for all sheaf quotients}\ E\twoheadrightarrow Q \}, \\
   \Fc_{B,\omega} &= \{ E \in \Coh (X) \colon \mu_{B,\omega}(A)\leq 0\ \text{for all subsheaves}\ A\hookrightarrow E\}.
\end{align*}
These subcategories form a torsion pair and so the extension closure in $D^b(X)$
\[
  \Coh^{B,\omega}(X) = \langle \Fc_{B,\omega}[1], \Tc_{B,\omega} \rangle,
\]
is the heart of a t-structure; in particular, it is an abelian category.  The central charge 
\[
Z_{B,\omega}(E) = -\int_X e^{-(B+i\omega)}\ch(E) = -\ch_2^B(E) + \frac{\omega^2}{2}\ch_0^B(E) + i \omega\ch_1^B(E)
\]
can then be paired with $\Coh^{B,\omega}(X)$ to form a Bridgeland stability condition $\sigma_{B,\omega} = (Z_{\omega,B}, \Coh^{B,\omega}(X))$ (see \cite{ABL} and also  \cite{SCK3} for the positivity and HN properties, and \cite[Theorem 3.23]{Toda2012StabilityCA} for the support property).
\begin{rem}
As customary, we will omit $\Gamma$ from the notation when $\Gamma=\widetilde{\mathrm{NS}}(X)$ and $v$ is the Chern character map.
\end{rem}

\begin{rem}
The stability conditions $\sigma_{B,\omega}$ are examples of \emph{geometric} stability conditions, i.e., stability conditions for which all the skyscraper sheaves $\mathcal{O}_x$ are stable of the same phase.
\end{rem}

Let us now consider the slightly more general central charges 
\begin{equation}\label{geometriccentralcharge}
Z_{a,B,\omega}=-\ch_2^B+a\ch_0^B+i\omega\ch_1^B,
\end{equation}
where $a \in \RR$.

\begin{lem}\label{lem:ZaBom}
For any $a>0$ and $\RR$-divisors $B,\omega$ where $\omega$ is ample, the pair $(Z_{a,B,\omega}, \Coh^{B,\omega}(X))$ is a geometric Bridgeland stability condition.
\end{lem}

\begin{proof}
Let us fix $a>0, \omega, B$ as in the statement of the lemma.  Let $t>0$ satisfy   $a = \tfrac{(t\omega)^2}{2}$ so that we have
\[
 \begin{pmatrix} 1 & 0 \\ 0 & \tfrac{1}{t} \end{pmatrix} Z_{B,t\omega} = Z_{a,B,\omega}.
\]
Now  let $g$  denote the lift $\left(\begin{pmatrix} 1 & 0 \\ 0 & \tfrac{1}{t} \end{pmatrix}, f\right)$ of $\begin{pmatrix} 1 & 0 \\ 0 & \tfrac{1}{t} \end{pmatrix} $ in $\wt{\mathrm{GL}}^+\!(2,\RR)$ such that $f$ fixes every element in $\tfrac{1}{2}\ZZ$, and let $\Pc$ denote the slicing for the Bridgeland stability $(Z_{B,t\omega},\Coh^{B,t\omega}(X))$ so that $\Pc(0,1] = \Coh^{B,t\omega}(X)$. Then $(Z_{B,t\omega}, \Coh^{B,t\omega}(X))\cdot g^{-1}$ is the Bridgeland stability $(Z',\Pc')$ where $Z'=Z_{a,B,\omega}$ and 
\[
\Pc' (0,1] = \Pc(f(0),f(1)]=\Pc (0,1]=\Coh^{B,t\omega}(X)=\Coh^{B,\omega}(X).
\]
Hence $(Z_{a,B,\omega},\Coh^{B,\omega}(X))$ is a Bridgeland stability condition.  Since the $\wt{\mathrm{GL}}^+\!(2,\RR)$-action on the stability manifold preserves geometric stability conditions, we are done.
\end{proof}

\begin{rem}
By \cite[Theorem 7.25 and Example 7.27]{Alper2018ExistenceOM}, it follows that if $B$ and $\omega$ are $\mathbb{Q}$-divisors, $a\in\mathbb{R}$, and $v$ is the Chern character of a $\sigma_{a,B,\omega}$-semistable object in $\Coh^{B,\omega}(X)$, then there exists a proper good moduli space $\mathcal{M}_{a,B,\omega}(v)$, which is an algebraic space over $\mathbb{C}$. See also \cite[Theorem 21.24]{StabFam} for a complete proof in the relative setting. 
\end{rem}

\begin{prop}\label{isoUpsilon}
Let $\Upsilon\in \Aut(D^b(X))$ and suppose that $\sigma=(Z,\mathcal{A})$ is a stability condition such that 
\begin{equation}
Z_{a,B,\omega}(\Upsilon(\_))=T Z(\_)
\end{equation}
for some $a \in \mathbb{R}_{>0}$, $\mathbb{R}$-divisors $\omega, B$ where $\omega$ is ample, and $T\in \mathrm{GL}^+(2,\mathbb{R})$. If $\Upsilon\cdot \sigma$ is a geometric stability condition, then 
\begin{equation}\label{eq:stabeq-1}
\Upsilon \cdot \sigma = \sigma_{a,B,\omega}\cdot g
\end{equation}
for some $g \in \widetilde{\mathrm{GL}}^+\!(2,\mathbb{R})$.  Furthermore,  there exists an integer $j$  depending on $\sigma, \Upsilon, B, \omega, T$, such that the autoequivalence $\Upsilon [j]$ induces an isomorphism
$$
\mathcal{M}_{\sigma}(v)\cong \mathcal{M}_{a,B,\omega}((-1)^j \Upsilon(v)).
$$
of good moduli spaces.
\end{prop}
\begin{proof}
By assumption, we have $Z_{a,B,\omega}(\_)= TZ(\Upsilon^{-1}(\_))$.  Since $(Z_{a,B,\omega},\Coh^{B,\omega}(X))$ is a geometric Bridgeland stability condition by Lemma \ref{lem:ZaBom}, there exists a lift $g=(T,f) \in \wt{\mathrm{GL}}^+\!(2,\RR)$ of $T$ such that $\Upsilon \cdot \sigma \cdot g^{-1}$, which has $TZ\Upsilon^{-1}$ as the central charge, has the property that all skyscraper sheaves $\OO_x$ are stable of phase 1 with respect to it.  Writing  $\Upsilon \cdot \sigma \cdot g^{-1}=(Z',\Pc')$, we have $Z'=Z_{a,B,\omega}$ while $\Pc'(0,1]$ coincides with $\Coh^{B,\omega}(X)$ by Lemma \ref{lem:MSLem6-20} below.  That is, we have $\Upsilon \cdot \sigma \cdot g^{-1} = \sigma_{a, B, \omega}$, giving us --.

Let $\phi$ denote the phase function associated to the Bridgeland stability codition $\sigma$.  For an object $E \in D^b(X)$, we now have
\begin{align*}
&E \in \Pc (0,1] \text{ and } \ch(E) = v \\
  &\Leftrightarrow \Upsilon (E) \text{ is $(\Upsilon\cdot \sigma)$-semistable of phase $\phi_\sigma (E)$, and } \ch(\Upsilon (E))=\Upsilon (v) \\
  &\Leftrightarrow \Upsilon (E) \text{ is  $(\Upsilon\cdot \sigma\cdot g^{-1})$-semistable of phase $f(\phi_\sigma (E))$, and } \ch(\Upsilon (E))=\Upsilon (v) \\
  &\Leftrightarrow \Upsilon (E)[j] \text{ is $(Z_{a,\omega,B},\Coh^{B,\omega}(X))$-semistable of phase $f(\phi_\sigma (E))+j$, and }   \ch(\Upsilon (E)[j])=(-1)^j\Upsilon (v)
\end{align*}
where $j$ is the unique integer such that $f(\phi_\sigma (E))+j \in (0,1]$.  Since by definition, the action of $\Upsilon$ respects S-equivalence classes, then the isomorphism of moduli spaces follows.
\end{proof}

\begin{rem}
From the proof of Proposition \ref{isoUpsilon}, we see that the choice of the integer $j$ ensures $\Upsilon (E) [j]$ lies in the heart $\Coh^{B,\omega}(X)$.  Since every object in the heart $\Coh^{B,\omega}(X)$ satisfies $\omega \ch_1^B \geq 0$, the sign $(-1)^j$ coincides with the sign of  $\omega \ch_1^B(\Upsilon (E))$.
\end{rem}

The following lemma tell us that the imaginary part of the central charge of a geometric stability condition completely determines its heart. This was originally stated for the central charges $Z_{B,\omega}$, but the proof goes through without change under slightly more general hypotheses:

\begin{lem}\cite[Lemma 6.20]{MSlec}\label{lem:MSLem6-20}
Let $(Z,\Ac)$ be a geometric stability condition such that all skyscraper sheaves have phase 1.  If  $\Im Z = \omega \ch_1^B$ for some ample divisor $\omega$ and $\RR$-divisor $B$, then $\Ac = \Coh^{B,\omega}(X)$.
\end{lem}

\section{Estimating the Gieseker chamber}

\begin{defn}
Let $B$ be a $\mathbb{Q}$-divisor. A torsion-free sheaf $E$ is called $B$-twisted Gieseker semistable with respect to an ample class $\omega$ if and only if
\begin{itemize}
    \item[(i)] $E$ is $\mu_\omega$-semistable, and
    \item[(ii)] For every subsheaf $A\hookrightarrow E$ such that $\mu_\omega(A)=\mu_{\omega}(E)$ we have

    $$
    \frac{\chi(A\otimes L )}{\ch_0(A)} \leq \frac{\chi(E\otimes L)}{\ch_0(E)},
    $$
    where $L=-B+\frac{K_X}{2}$ and $\chi(\_\otimes L)$ is computed using the Hirzebruch-Riemann-Roch formula.
\end{itemize}




Likewise, a pure 1-dimensional sheaf $\mathcal{E}$ is called $B$-twisted Gieseker semistable with respect to $\omega$ if and only if for every subsheaf $\mathcal{F}\hookrightarrow \mathcal{E}$ we have
$$
\frac{\chi(\mathcal{F}\otimes L)}{\omega c_1(\mathcal{F})}\leq \frac{\chi(\mathcal{E}\otimes L)}{\omega c_1(\mathcal{E})}.
$$
\end{defn}

\begin{rem}
The notion of twisted Gieseker semistability was introduced by Matsuki and Wentworth in \cite{MW}, where they proved the existence of projective moduli spaces parametrizing S-equivalence classes of $B$-twisted Gieseker semistable sheaves, of a given Chern character $v$ with respect to a given polarization $\omega$. We will denote these spaces by $M_{B,{\omega}}(v)$. Notice that with our notation $M_{{K_X/2},{\omega}}(v)$ is the usual Gieseker moduli space.
\end{rem}

\begin{rem}\label{largevolumelimit} As proven by Bridgeland \cite[Prop.\  14.2]{SCK3} and Lo-Qin \cite[Theorems 1.1, 1.2(i)]{LQ}, the only objects with non-negative rank and positive degree (with respect to some polarization $\omega$ and $\mathbb{R}$-divisor $B$) that remain $\sigma_{a,B,\omega}$-semistable for large values of the parameter $a$ are precisely the $B$-twisted Gieseker semistable sheaves with respect to $\omega$. Together with the estimates of Lemma \ref{estimates a-mini-walls} below, this can be rephrased as follows: Given a Chern character $v$ satisfying $\ch_0(v)>0$, $\omega\ch_1^B(v)>0$ and $\Delta(v)\geq 0$, then for $a\gg 0$
$$
\mathcal{M}_{a,B,
\omega}(v)\cong M_{B,\omega}(v).
$$
\end{rem}

We can also prove a negative rank version of Remark \ref{largevolumelimit}.  Let us define  $E^D=R\mathcal{H}om(E,\mathcal{O})[1]$ for every $E\in D^b(X)$.

\begin{prop}\label{duality}
Let $v$ be a Chern character satisfying $\ch_0(v)<0$, $\omega\ch_1^B(v)>0$, and $\Delta(v)\geq 0$. Then there exists $a\gg 0$ such that
$$
\mathcal{M}_{a,B,\omega}(v)\cong M_{-B,\omega}(v^D).
$$
\end{prop}

\begin{proof}
Since $\omega\ch_1^B(v)>0$ then every object $E\in \mathcal{B}_{\omega,B}$ with $\ch(E)=v$ has phase in $(0,1)$. Thus, by \cite[Theorem 4.1]{martinez2017duality} the duality functor induces an isomorphism 
$$
\mathcal{M}_{a,B,\omega}(v)\cong \mathcal{M}_{a,-B,\omega}(v^D).
$$
Now, notice that  $\ch_0(v^D)=-\ch_0(v)>0$, $\omega\ch_1^{-B}(v^D)=\omega\ch_1^B(v)>0$, and $\Delta(v^D)=\Delta(v)\geq 0$. Therefore, Remark \ref{largevolumelimit} implies that for $a\gg 0$ we have an isomorphism $\mathcal{M}_{a,-B,\omega}(v^D)\cong M_{-B,\omega}(v^D)$,
and so $\mathcal{M}_{a,B,\omega}(v)\cong M_{-B,\omega}(v^D)$.
\end{proof}

\begin{rem}\label{rem:BGineq}
Given $\RR$-divisors $B, \omega$ on $X$, the corresponding discriminant $\Delta_{B,\omega}$ of an object $E \in D^b(E)$ is
\[
  \Delta_{B,\omega}(E) := (\omega \ch_1^B(E))^2 - 2\omega^2(\ch_0^B(E))(\ch_2^B(E)).
\]
When $\omega$ is ample, every $\sigma_{B,\omega}$-semistable object $E$ satisfies the Bogomolov-Gieseker inequality
\[
  \Delta_{B,\omega}(E) \geq 0
\]
(see \cite[Theorem 6.13]{MSlec}).  From the proof of Lemma \ref{lem:ZaBom}, we know that a Bridgeland stability of the form $(Z_{a,B,\omega}, \Coh^{B,\omega}(X))$ differs from a Bridgeland stability of the form $\sigma_{B', \omega'}$ by a $\wt{\mathrm{GL}}^+\!(2,\RR)$-action, which does not change the set of semistable objects.  As a result, any $(Z_{a,B,\omega}, \Coh^{B,\omega}(X))$-semistable object also satisfies the Bogomolov-Gieseker inequality.
\end{rem}
\begin{lem}\label{equalslope}
If $E$ is a $B$-twisted $\omega$-semistable sheaf of positive rank and $A\hookrightarrow E$ is a subobject of $E$ in $\Coh^{B,\omega}(X)$ such that $\mu_{B,\omega}(A)=\mu_{B,\omega}(E)$ and $\phi_{a,B,\omega}(A)=\phi_{a,B,\omega}(E)$, then $A$ is a subsheaf of $E$ that makes $E$ a properly $B$-twisted $\omega$-semistable sheaf. In particular, $A$ can never destabilize $E$ for any value of $a$.  
\end{lem}

\begin{proof}
Consider the short exact sequence $0\rightarrow A\rightarrow E\rightarrow B\rightarrow 0$ in $\Coh^{B,\omega}(X)$. This sequence induces the long exact sequence in $\Coh(X)$:
$$
0\rightarrow \mathcal{H}^{-1}(A)\rightarrow 0\rightarrow \mathcal{H}^{-1}(B)\rightarrow \mathcal{H}^0(A)\rightarrow E\rightarrow \mathcal{H}^0(B)\rightarrow 0.
$$
In particular, $\mathcal{H}^{-1}(A)=0$ and so $A$ is quasi-isomorphic to the sheaf $\mathcal{H}^0(A)$. The sheaf quotient $L:= A/\mathcal{H}^{-1}(B)$ fits into a diagram
$$
\xymatrix{
0\ar[r] &\mathcal{H}^{-1}(B)\ar[r] & A\ar[rr]\ar@{->>}[rd] & & E\ar[r] &\mathcal{H}^0(B)\ar[r]& 0.\\
& & & L\ar@{^{(}->}[ur]& & &
}
$$
If $\mathcal{H}^{-1}(B)$ is nonzero, then $\mu_{B,\omega}(\mathcal{H}^{-1}(B))\leq 0<\mu_{B,\omega}(A)$ because $\mathcal{H}^{-1}(B)\in \mathcal{F}_{B,\omega}$ and $A\in\mathcal{T}_{B,\omega}$. Thus, 
$$
\mu_{B,\omega}(E)=\mu_{B,\omega}(A)<\mu_{B,\omega}(L),
$$
contradicting the semistability of $E$. Therefore $\mathcal{H}^{-1}(B)=0$ and $A$ is a subsheaf of $E$. On the other hand, since $\phi_{a,B,\omega}(A)=\phi_{a,B,\omega}(E)$ we have that
\begin{equation}\label{forRemarkonGenericity}
\frac{\displaystyle\frac{\ch_2^{B}(A)}{\ch_0(A)}-a}{\mu_{B,\omega}(A)}=\frac{\displaystyle\frac{\ch_2^{B}(E)}{\ch_0(E)}-a}{\mu_{B,\omega}(E)},
\end{equation}
which is independent of $a$ because $\mu_{B,\omega}(A)=\mu_{B,\omega}(E)$. Therefore,  $\displaystyle \frac{\ch_2^{B}(A)}{\ch_0(A)}=\frac{\ch_2^{B}(E)}{\ch_0(E)}$ and $E$ is properly $B$-twisted $\omega$-Gieseker semistable.
\end{proof}
\begin{rem}\label{genericity}
Suppose that $E\in \Coh^{B,\omega}(X)$ is a properly $B$-twisted $\omega$-semistable sheaf of Chern character $v$. Then there is a subsheaf $A\hookrightarrow E$ such that
\begin{equation}\label{properlyGieseker}
\mu_{B,\omega}(A)=\mu_{B,\omega}(E)\ \ \text{and}\ \ \frac{\ch_2^B(A)}{\ch_0(A)}=\frac{\ch_2^B(E)}{\ch_0(E)},
\end{equation}
in which case the equality in equation \eqref{forRemarkonGenericity} is satisfied. Since $E$ is slope semistable then so is every subsheaf of the same slope, thus $A$ is a $\sigma_{a,B,\omega}$-destabilizing subobject of $E$ for all $a>0$. In other words, the 1-parameter family of stability conditions $\{\sigma_{a,B,
\omega}\}_{a>0}$ is contained on a Bridgeland wall. Conversely, if $A\hookrightarrow E$ is a subobject of $E$ in $\Coh^{B,\omega}(X)$ such that $\phi_{a,B,\omega}(A)=\phi_{a,B,\omega}(E)$ for all $a>a_0$ then Equation \eqref{forRemarkonGenericity} implies that the equalities in \eqref{properlyGieseker} hold. Moreover, such $A$ is a sheaf and the quotient $T=E/A\in \Coh^{B,\omega}(X)$ is  $\sigma_{a,B,\omega}$-semistable for all $a\gg 0$ (where $E$ is semistable) and so it must be a sheaf as well because otherwise
$$
\lim_{a\rightarrow +\infty}\frac{\ch_2^B(\mathcal{H}^{-1}(T)[1])-a\ch_0(\mathcal{H}^{-1}(T)[1])}{\omega\ch_1^B(\mathcal{H}^{-1}(T)[1])}=+\infty.
$$
Therefore, $A$ is a subsheaf destabilizing $E$ with respect to $B$-twisted $\omega$-Gieseker semistability. Thus, $\{\sigma_{a,B,\omega}\}_{a>a_0}$ is contained in a Bridgeland chamber for $v$ if and only if there are not properly $B$-twisted $\omega$-semistable sheaves of Chern character $v$.    
\end{rem}

\begin{lem}\label{estimates a-mini-walls}
Let $B,\omega\in \mathrm{NS}(X)_{\mathbb{Q}}$ with $\omega$ ample. Let $E$ be a $B$-twisted $\omega$-Gieseker semistable sheaf with $\ch_0(E)>0$ and $\omega\ch_1^B(E)> 0$.  Then $E$ is $\sigma_{a,B,\omega}$-semistable for 
$$
a> \frac{n}{2}\mu_{B,\omega}(E)\Delta_{B,\omega}(E),
$$
where $n\in \mathbb{N}$ is chosen so that $n (\omega\ch_1(C)) \omega^2\in\mathbb{Z}$ for all $C\in D^b(X)$.  When this inequality is attained, the moduli space of  $B$-twisted $\omega$-Gieseker semistable sheaves coincides with the moduli space of   $\sigma_{a, B, \omega}$-semistable objects.
\end{lem}

\begin{proof}
We only need to find a bound $a_0$, depending only on the Chern character of $E$, to ensure that $E$ is not $\sigma_{a,B,\omega}$-unstable for any $a>a_0$.

First, suppose that $E$ is $\mu_\omega$-semistable with $\mu_{B,\omega}(E)>0$, then as in the proof of Lemma \ref{equalslope}, the long exact sequence of cohomology for a destabilizing sequence $0\rightarrow A\rightarrow E\rightarrow Q\rightarrow 0$ in $\Coh^{B,\omega}(X)$ gives
$$
0\rightarrow H^{-1}(Q)\rightarrow A\rightarrow E\rightarrow H^0(Q)\rightarrow 0.
$$
Since $E$ is a torsion-free sheaf, this shows that $A$ is a sheaf of positive rank  and, moreover, $\mu_{B,\omega}(A)\leq \mu_{B,\omega}(E)$. Indeed, if $\mu_{B,\omega}(A)>\mu_{B,\omega}(E)$ and $A_1\hookrightarrow A$ is the first Harder-Narasimhan factor of $A$, then $\mu_{B,\omega}(A_1)>\mu_{B,\omega}(E)$ and so $A_1$ is a subsheaf of $H^{-1}(Q)$. This is a contradiction unless $H^{-1}(Q)=0$ since $H^{-1}(Q)\in \mathcal{F}_{B,\omega}$.

Since by Lemma \ref{equalslope}, subobjects with $\mu_{B,\omega}(A)=\mu_{B,\omega}(E)$ can never make $E$ unstable, we can further assume that $\mu_{B,\omega}(A)<\mu_{B,\omega}(E)$.

Notice that since $\omega\ch_1^B(E)>0$ and $Q$ is a destabilizing quotient of $E$ we must have $\omega\ch_1^B(Q)>0$ and so 
\begin{equation*}
0<\omega\ch_1^B(A)<\omega\ch_1^B(E)\ \ \text{and hence}\ \ \mu_{B,\omega}(A)<\omega\ch_1^B(E).
\end{equation*}

Now, the wall equation 
\[
 - \frac{\Re Z_{a,B,\omega}(A)}{\Im Z_{a,B,\omega}(A)}=- \frac{\Re Z_{a,B,\omega}(E)}{\Im Z_{a,B,\omega}(E)}
\]
becomes
$$
a(\mu_{B,\omega}(E)-\mu_{B,\omega}(A))=\frac{\ch_2^B(A)}{\ch_0(A)}\mu_{B,\omega}(E)-\frac{\ch_2^B(E)}{\ch_0(E)}\mu_{B,\omega}(A).
$$
Since we have the Bogomolov-Gieseker inequality $\Delta_{B,\omega}(A)\geq 0$ from Remark \ref{rem:BGineq}, we obtain
\begin{equation}\label{eq:c1}
a(\mu_{B,\omega}(E)-\mu_{B,\omega}(A))\leq \frac{1}{2\omega^2}\mu_{B,\omega}(A)^2\mu_{B,\omega}(E)-\frac{\ch_2^B(E)}{\ch_0(E)}\mu_{B,\omega}(A).
\end{equation}
Then 
\begin{align}
  n(\mu_{B,\omega}(E)-\mu_{B,\omega}(A)) &= n \left( \frac{\omega \ch_1^B(E)}{\ch_0(E)}-\frac{\omega \ch_1^B(A)}{\ch_0(A)}\right) \notag\\
  &= \frac{n}{\ch_0(E)\ch_0(A)}( \omega \ch_1(E)\ch_0(A) - \omega \ch_1(A)\ch_0(E)) \notag\\
  & \geq \frac{1}{\ch_0(E)\ch_0(A) \omega^2} \label{eq:c2}
\end{align}
where the last inequality follows from  our assumption on $n$, namely $n \omega \ch_1(M) \in \tfrac{1}{\omega^2}\mathbb{Z}$ for any $M \in D^b(X)$.

Combining \eqref{eq:c1} and \eqref{eq:c2}, we obtain
\begin{align*}
a&\leq \frac{n}{2}\frac{(\omega\ch_1^B(A))^2\cdot \omega\ch_1^B(E)}{\ch_0(A)}-n\omega^2\ch_2^B(E)\cdot \omega\ch_1^B(A)\\
&= \frac{n}{2} \mu_{B,\omega}(A)\cdot  \omega \ch_1^B(A)\cdot  \omega \ch_1^B (E) - n \omega^2\ch_2^B(E)\cdot \omega\ch_1^B(A) \\
&\leq \frac{n}{2}\mu_{B,\omega}(E) \cdot  \omega \ch_1^B(A)\cdot  \omega \ch_1^B (E) - n \omega^2\ch_2^B(E)\cdot \omega\ch_1^B(A)\\
&= \frac{n}{2} \omega \ch_1^B(A) \left( \frac{(\omega\ch_1^B(E))^2 - 2\omega^2 \ch_2^B(E)\ch_0(E)}{\ch_0(E)}\right)\\
&< \frac{n}{2} \omega\ch_1^B(E)  \left( \frac{(\omega\ch_1^B(E))^2 - 2\omega^2 \ch_2^B(E)\ch_0(E)}{\ch_0(E)}\right) \\
&= \frac{n}{2} \mu_{B,\omega}(E)\Delta_{B,\omega}(E)
\end{align*}
where, in the last inequality, we use the fact that $\Delta_{B,\omega}(E) \geq 0$ because $E$ is a $B$-twisted Gieseker semistable sheaf.  Thus, such a destabilizing sequence does not exist when
$$
a>\frac{n}{2}\mu_{B,\omega}(E)\Delta_{B,\omega}(E).
$$
\end{proof}

\begin{lem}\label{one dim real estimates}
 Let $B,\omega\in\mathrm{NS}(X)_{\mathbb{R}}$ with $\omega$ ample. If $\mathcal{E}$ is a 1-dimensional $B$-twisted Gieseker semistable sheaf, then $\mathcal{E}$ is $\sigma_{a,B,\omega}$-semistable for
$$
a>\max\left\{\frac{(\omega c_1(\mathcal{E}))^2}{2\omega^2},\ \frac{(\omega c_1(\mathcal{E}))^2}{2\omega^2}-\ch_2^B(\mathcal{E})\right\}.
$$
\end{lem}

\begin{proof}
Suppose that $A\hookrightarrow \mathcal{E}$ is a destabilizing subobject in $\Coh^{B,\omega}(X)$ producing a wall.  If $A$ has rank zero, then it is a subsheaf of $\mathcal{E}$, in which case the twisted Gieseker semistability of $\mathcal{E}$ prevents $A$ from being a destabilizing subobject in $\Coh^{B,\omega}(X)$.  Hence  $A$ is a sheaf of positive rank and so
$$
0<\mu_{B,\omega}(A)<\omega c_1(\mathcal{E}).
$$
The wall equation becomes
$$
\frac{\ch_2^B(A)-a\ch_0(A)}{\omega\ch_1^B(A)}=\frac{\ch_2^B(\mathcal{E})}{\omega c_1(\mathcal{E})},
$$
which gives
\begin{align*}
    a&=\frac{\ch_2^B(A)}{\ch_0(A)}-\mu_{B,\omega}(A)\frac{\ch_2^B(\mathcal{E})}{\omega c_1(\mathcal{E})}\\
    &\leq \frac{1}{2\omega^2}\left(\frac{\omega\ch_1^B(A)}{\ch_0(A)}\right)^2-\mu_{B,\omega}(A)\frac{\ch_2^B(\mathcal{E})}{\omega c_1(\mathcal{E})}\ \ \ \ \ \  (\text{because}\ \ \Delta_{B,\omega}(A)\geq 0)\\
    &=\frac{1}{2\omega^2}(\mu_{B,\omega}(A))^2-\mu_{B,\omega}(A)\frac{\ch_2^B(\mathcal{E})}{\omega c_1(\mathcal{E})}\\
    &\leq \max\left\{\frac{(\omega c_1(\mathcal{E}))^2}{2\omega^2},\ \frac{(\omega c_1(\mathcal{E}))^2}{2\omega^2}-\ch_2^B(\mathcal{E})\right\}.
\end{align*}
\end{proof}


From Remark \ref{largevolumelimit}, we know that a $\sigma_{a,B,\omega}$-semistable object for $a \gg 0$ is necessarily a $B$-twisted Gieseker semistable sheaf with respect to $\omega$. Along  with Lemmas \ref{estimates a-mini-walls} and \ref{one dim real estimates}, we now have:

\begin{cor}\label{GiesekerChamber} Let $v$ be a Chern character with $\ch_0(v) \neq 0$,  $\Delta(v)\geq 0$, and let $B,\omega\in \mathrm{NS}(X)_{\mathbb{Q}}$ be classes with $\omega$ ample and satisfying $\omega \ch_1^{B}(v)\neq 0$. Then there exists $a_0>0$ such that for all $a>a_0$, the moduli space of $\sigma_{a,B,\omega}$-semistable objects of Chern character $v$ is isomorphic to a moduli space of (twisted) Gieseker semistable sheaves. 
\end{cor}

\begin{proof}
If the signs of $\ch_0(v)$ and $\omega\ch_1^B(v)$ are equal then, up to a shift, we can assume that $\ch_0(v)>0$ and by Lemma \ref{estimates a-mini-walls} we know that the moduli space $\mathcal{M}_{a,B,\omega}(v)$ is isomorphic to $M_{B,\omega}(v)$ for all $a>a_0$, where 
$$
a_0=\frac{n}{2}\mu_{B,\omega}(v)\Delta_{B,\omega}(v).
$$
If the signs of $\ch_0(v)$ and $\omega\ch_1^B(v)$ are different then, up to a shift, we can assume that $\ch_0(v)<0$ and by Proposition \ref{duality} we know that the moduli space $\mathcal{M}_{a,B,\omega}(v)$ is isomorphic to $M_{-B,\omega}(v^D)$ for all $a>a_0$,  where
$$
a_0=-\frac{n}{2}\mu_{B,\omega}(v)\Delta_{B,\omega}(v).
$$
\end{proof}

Even though the bounds obtained in Lemma \ref{estimates a-mini-walls} and \ref{one dim real estimates} are easy to write, they can be improved a lot in special cases. For instance, in the case of Weierstra\ss\ elliptic surfaces and polarizations of the form $\omega=\Theta+bf$, if $f\ch_1^B(v)=0$ we can obtain a bound that is independent of $\omega$. The precise result is the following.

\begin{lem}\label{estimates_0fiberdegree}
Let $X$ be a Weierstra\ss\ elliptic surface and let $v$ be a Chern character with $\ch_0(v)> 0$,  $f\ch_1(v)\geq 0$ and $\Delta(v)\geq 0$. Let $\lambda \in \mathbb{Q}, b \in \mathbb{Q}_{>0}$ be such that $$\omega\ch_1^B(v)>0$$ holds if we put $\omega=\Theta+bf$ and $B=\mu_f(v)\Theta+\lambda f$ where $\mu_f(v)=f\ch_1(v)/\ch_0(v)$. Then there exists $n\in\mathbb{N}$ such that for
$$
a>\max\left\{\frac{n(\Theta\ch_1^B(v))^3}{2\ch_0(v)\omega^2}, \frac{n(\Theta\ch_1^B(v))^3}{2\ch_0(v)\omega^2}-n\ch_2^B(v)\Theta\ch_1^B(v)\right\},
$$
the only $\sigma_{a,B,\omega}$-semistable objects of Chern character $v$ are precisely the $B$-twisted Gieseker semistable sheaves with respect to $\omega$. In particular, for $\omega$ integral this bound can be chosen independently of $\omega$.
\end{lem}

For a Chern character $v=(v_0, v_1, v_2)$ with $v_0 \neq 0$ and $B$ of the form $p\Theta + qf$ for some $p, q \in \mathbb{R}$, note that $f\ch_1^B(v)=0$ if and only if $fB=\mu_f(v)$.  The choice of $B$ in Lemma \ref{estimates_0fiberdegree} therefore  ensures that we have $f\ch_1^B(v)=0$.

\begin{proof}
Suppose $E$ is a $B$-twisted Gieseker semistable sheaf with respect to $\omega$.  As in the proof of Lemma \ref{estimates a-mini-walls},  if $A\hookrightarrow E$ is a $\sigma_{a,B,\omega}$-destabilizing object of a sheaf $E$ with $\ch(E)=v$, then $A\in\Coh(X)$ and 
\[
0<\omega\ch_1^B(A)<\omega\ch_1^B(E)=\Theta\ch_1^B(v),
\]
where the last equality holds because $f\ch_1^B(v)=0$ given our assumption on $B$; moreover, as in the proof of Lemma \ref{estimates a-mini-walls}, we can assume
\[
 \mu_{B,\omega}(A)<\mu_{B,\omega}(E).
\]
The wall equation can be written as
$$
a(\mu_{B,\omega}(A)-\mu_{B,\omega}(E))=\frac{\ch_2^B(v)}{\ch_0(v)}\mu_{B,\omega}(A)-\frac{\ch_2^B(A)}{\ch_0(A)}\mu_{B,\omega}(v).
$$
Since $\omega$ is a rational class,  there exists $n\in\mathbb{N}$ such that
$$
n(\omega\ch_1(C)) {\omega^2}\in\mathbb{Z},\ \ \text{for all}\ \ C\in D^b(X).
$$
In particular, if $\omega$ is integral we can just take $n=1$. Thus
$$
na(\mu_{B,\omega}(A)-\mu_{B,\omega}(E))\leq \frac{-a}{\ch_0(A)\ch_0(v)}.
$$
Therefore
$$
\frac{a}{\ch_0(A)\ch_0(v)}\leq n\left[\frac{\ch_2^B(A)}{\ch_0(A)}\mu_{B,\omega}(v)-\frac{\ch_2^B(v)}{\ch_0(v)}\mu_{B,\omega}(A)\right].
$$
Using that at the wall $A$ is also Bridgeland semistable we know that $\Delta_{B,\omega}(A)\geq 0$ and we get
\begin{align*}
a&\leq \frac{n(\omega\ch_1^B(A))^2}{2\ch_0(A)\omega^2}\Theta\ch_1^B(v)-n\ch_2^B(v)\omega\ch_1^B(A)\\
&< \frac{n(\Theta\ch_1^B(v))^3}{2\ch_0(v)\omega^2}-n\ch_2^B(v)\omega\ch_1^B(A)\\
&\leq \max\left\{\frac{n(\Theta\ch_1^B(v))^3}{2\ch_0(v)\omega^2},\frac{n(\Theta\ch_1^B(v))^3}{2\ch_0(v)\omega^2}-n\ch_2^B(v)\Theta\ch_1^B(v)\right\}
\end{align*}
where the last inequality follows by considering the sign of $\ch_2^B(v)$.  
Therefore, for $a$ greater than the given bound, no $B$-twisted Gieseker semistable sheaf with respect to $\omega$ can be destabilized.  Since we know from Remark \ref{largevolumelimit}  that the moduli space of $\sigma_{a, B, \omega}$-semistable objects coincides with the moduli space of $B$-twisted Gieseker semistable sheaves with respect to $\omega$ for $a \gg 0$, it follows that the two moduli spaces coincide when $a$ is greater than the bound above.
\end{proof}

\section{Preserving twisted Gieseker stability}\label{preservation:gieseker}
From now on, $X$ will denote a Weierstra\ss\  elliptic surface with canonical section $\Theta$ and fiber $f$. The aim of this section is to use Proposition \ref{isoUpsilon} to relate different Bridgeland moduli spaces over $X$  via the relative Fourier-Mukai transform $\Phi$. For this purpose we will analyze the image of the submanifold $\Stab^{\Gamma}(X)\subset \Stab(X)$ under the action of $\Phi$, where $\Gamma=\mathbb{Z}\oplus\mathbb{Z}\oplus\mathbb{Z}\oplus\frac{1}{2}\mathbb{Z}$ and $v$ assigns every object in $E\in D^b(X)$ its Chern character table, defined by
$$
v(E)=\begin{array}{|c|c|}\hline \ch_0(E) & f\ch_1(E)\\ \hline \Theta\ch_1(E)& \ch_2(E)\\ \hline \end{array}.
$$
More precisely, set 
\[
\bar{B}=\bar{p}\Theta+\bar{q}f, \text{\quad} \bar{\omega}=\Theta+\bar{b}f,
\]
and $\bar{a}$ such that $Z_{\bar{a},\bar{B},\bar{\omega}}$ is the central charge of a stability condition. We want to find 
\[
B=p\Theta+qf, \text{\quad} \omega=\Theta+bf,
\]
and $a$ such that
\begin{equation}\label{eq:mainZaBwT}
Z_{a,B,\omega}(\Phi(\_))=T Z_{\bar{a},\bar{B},\bar{\omega}}(\_)
\end{equation}
for some $T\in \mathrm{GL}^+(2,\mathbb{R})$. Solving this central charge equation together with Lemma \ref{lem:MSLem6-20} will give the  equation of Bridgeland stability conditions
\begin{equation}\label{eq:stabeq-2}
\Phi \cdot \sigma_{\bar{a}, \bar{B}, \bar{\omega}} = \sigma_{a, B, \omega}\cdot g 
\end{equation}
for some $g \in \widetilde{\mathrm{GL}}^+\!(2,\mathbb{R})$.  That is, we will obtain a  class of examples of Proposition \ref{isoUpsilon}, different from those already obtained in  \cite[Theorem 11.7]{Lo20} which deals with Bridgeland stability conditions of the form $\sigma_{B,\omega}$ where $B$ is a multiple of the fiber class.



We can now solve for the unbarred parameters $a, b, p, q$ in terms of the barred parameters $\bar{a}, \bar{b}, \bar{p}, \bar{q}$ in  equation \eqref{eq:mainZaBwT}  (see Appendix \ref{computations}).  Assuming $\bar{a}>0$ and $\bar{b} \geq e$, we obtain  $b>e$ and hence $2b-e>0$, and

\begin{align}\label{usual:coordinates}
T&=\begin{pmatrix}-p & 1+p\bar{p}\\ -1 & \bar{p}\end{pmatrix}\notag\\
b&=(\bar{a}+e)+\bar{p}^2\left(\bar{b}-\frac{e}{2}\right)\notag\\
p&=\frac{-\bar{p}(2\bar{b}-e)}{2b-e}\\
q&=\bar{q}+\frac{e}{2}(p-\bar{p}-1)\notag\\
a&=(\bar{b}-e)-p^2\left(b-\frac{e}{2}\right).\notag
\end{align}

In what follows, it will be convenient  to rewrite equations \eqref{usual:coordinates} using the  coordinates 
\[
\bar{V}:= \bar{\omega}^2=2\bar{b}-e, \text{\quad} V:=\omega^2=2b-e,
\]
which we call \emph{volume coordinates}, and the  coordinates
\[
\bar{U}=2\bar{a}+e, \text{\quad} U=2a+e,
\]
which we call \emph{volume-like\footnote{We refer to this coordinates as volume-like coordinates because in the standard form \cite{ABL} for geometric stability conditions we have $2a=\omega^2$.} coordinates.}  
Note that $V=2b-e>0$ under our assumptions on $\bar{a}, \bar{b}$.  We can now easily rewrite  \eqref{usual:coordinates} as
\begin{align}
    V &= \bar{U}+\bar{p}^2\bar{V} \label{eq:V}\\
    p &= -\bar{p}\frac{\bar{V}}{V} \label{eq:p}\\
    q&=\bar{q}+\frac{e}{2}(p-\bar{p}-1)\label{eq:q}\\
    U &= \bar{V}-p^2V \label{eq:Ut1}.
\end{align}


In fact, from \eqref{eq:Ut1}, \eqref{eq:V}, \eqref{eq:p}, we have
\begin{align*}
  UV &= \bar{V}V-p^2V^2, \\
  \bar{U}\bar{V} &= (V-\bar{p}^2\bar{V})\bar{V} = V\bar{V} - \bar{p}^2 \bar{V}^2, \\
  p^2V^2 &= \bar{p}^2\bar{V}^2,
\end{align*}
respectively. Putting these three equations together, we obtain
\begin{equation}\label{eq:U}
UV = \bar{U}\bar{V}.
\end{equation}
If we further assume $\bar{V}>0$, then a set of solutions for the equation
$$
T' Z_{a,B,\omega}(\_)= Z_{\bar{a},\bar{B},\bar{\omega}}(\whPhi(\_)[1])
$$
is given by 
\begin{align}
T'&=T^{-1}=\begin{pmatrix}\bar{p} & -1-p\bar{p}\\ 1 & -{p}\end{pmatrix}\notag\\
    \bar{V}&={U}+{p}^2{V}\label{eq:barV}\\
    \bar{p}&=-{p}\frac{{V}}{\bar{V}}\label{eq:barp}\\
    \bar{q}&={q}+\frac{e}{2}(\bar{p}-{p}+1)\label{eq:barq}\\
    \bar{U}&=V-\bar{p}^2\bar{V}\label{eq:barU}
\end{align}
which further exemplifies the symmetry of our solutions.

\begin{prop}\label{prop:anaughtbound}
Given $\bar{\omega}=\Theta+\bar{b}f$ ample ($\bar{b}>e$), and $\bar{B}=\bar{p}\Theta+\bar{q}f$, there exists $\bar{a}_0>0$ such that, for all $\bar{a}>\bar{a}_0$, the stability condition $\Phi\cdot \sigma_{\bar{a},\bar{B},\bar{\omega}}$ is  geometric.
\end{prop}

\begin{proof}
In order to conclude $\Phi \cdot \sigma_{\bar{a}, \bar{B}, \bar{\omega}}$ is a geometric stability condition, we need to show that all the skyscraper sheaves $\OO_x$ are $\Phi  \cdot \sigma_{\bar{a}, \bar{B}, \bar{\omega}}$-stable of the same phase.  In particular, $\OO_x$ being $\Phi  \cdot \sigma_{\bar{a}, \bar{B}, \bar{\omega}}$-stable is equivalent to $\Phi^{-1}(\OO_x)$ being $\sigma_{\bar{a}, \bar{B}, \bar{\omega}}$-stable.  Since $\Phi^{-1} \cong \whPhi [1]$, the object $\Phi^{-1}(\OO_x)$ is merely a shift of a  rank-one, torsion-free sheaf on a fiber of the elliptic fibration.  Therefore, it suffices  to find $\bar{a}_0$ such that every 1-dimensional Gieseker stable sheaf $\mathcal{E}$ on $X$ with $\ch(\mathcal{E})=(0,f,0)$ is $\sigma_{\bar{a},\bar{B},\bar{\omega}}$-stable for all $\bar{a}>\bar{a}_0$. By Lemma \ref{one dim real estimates}, it is enough to take
\[
\bar{a}_0=\max\left\{ \frac{1}{2(2\bar{b}-e)},\ \frac{1}{2(2\bar{b}-e)}+\bar{p}\right\}.
\]
\end{proof}

\begin{cor}\label{cor:stabeq-1}
Let $\bar{\omega}, \bar{B}, \bar{a}$ be as in Proposition \ref{prop:anaughtbound}.  Then for any $\bar{a}>\bar{a}_0$, there exist $\mathbb{R}$-divisors $\omega, B$ with $\omega$ ample, and $g \in \widetilde{\mathrm{GL}}^+\! (2,\mathbb{R})$ such that the equation
\begin{equation}\label{eq:stabeq30}
  \Phi  \cdot \sigma_{\bar{a}, \bar{B}, \bar{\omega}} = \sigma_{a, B, \omega } \cdot g
\end{equation}
holds.
\end{cor}

\begin{proof}
Suppose $\bar{a} > \bar{a}_0$.  From the proof of Proposition \ref{prop:anaughtbound}, we know that for every skyscraper sheaf $\OO_x$, the object $\Phi^{-1}(\OO_x) \cong \whPhi \OO_x [1]$ is $\sigma_{\bar{a}, \bar{B}, \bar{\omega}}$-stable with phase in $(1,2]$.  Now choose a lift $g \in \widetilde{\mathrm{GL}}^+\!(2,\mathbb{R})$ of the matrix $T$ in \eqref{usual:coordinates} so that $\OO_x$ is $\Phi \cdot \sigma_{\bar{a}, \bar{B}, \bar{\omega}} \cdot g^{-1}$-stable of phase in $(0,1]$.  Since we know the central charge of $\Phi \cdot \sigma_{\bar{a}, \bar{B}, \bar{\omega}} \cdot g^{-1}$ is precisely 
\[
  T Z_{\bar{a}, \bar{B}, \bar{\omega}} (\Phi^{-1}(\_)) 
\]
which equals $Z_{a, B, \omega}$ from \eqref{eq:mainZaBwT}, using Lemma \ref{lem:MSLem6-20} we can conclude 
\[
\Phi \cdot \sigma_{\bar{a}, \bar{B}, \bar{\omega}} \cdot g^{-1} = \sigma_{a, B, \omega }
\]
which gives \eqref{eq:stabeq30}.
\end{proof}

\begin{rem}\label{U-estimate-geom}
Notice that in Proposition \ref{prop:anaughtbound}, the assumption $\bar{b}>e$ implies $\bar{V}=2\bar{b}-e>\bar{b}>e$.  Thus in volume coordinates, the condition $\bar{a}>\bar{a}_0$ reads
$$
\bar{U}>\max\left\{\frac{1}{\bar{V}}+e,\ \frac{1}{\bar{V}}+e+2\bar{p}\right\}.
$$
Also, if  $\bar{p}\leq 0$ then we can take $\bar{a}_0=1/2e$. 
\end{rem}

\begin{rem}
When the $B$-field is a multiple of $f$, an analogue of  Proposition \ref{prop:anaughtbound} was proved in \cite[Theorem 11.7]{Lo20}.  There, to ensure that the functor $\Phi [1]$ takes  a geometric stability condition to a geometric stability condition, an integrality condition was imposed on the coefficients of $\bar{\omega}$ and $\bar{B}$.  In this article,  we use  Lemma \ref{one dim real estimates}   instead of the integrality condition.
\end{rem}

\begin{rem}\label{DiscriminantBaA}
As noticed in Corollary \ref{GiesekerChamber}, given divisors $B, \omega \in \mathrm{NS}(X)_{\mathbb{Q}}$ with $\omega$ ample, the existence of a Gieseker chamber for the Chern character $v$ can be ensured by requiring $\ch_0(v) \neq 0, \omega \ch_1^B(v) \neq 0$ and  $\Delta(v)\geq 0$.  Thus, one way to ensure the existence of a Gieseker chamber after the autoequivalence $\Phi[1]$ is to require 
$$
\Delta(\Phi(v)[1])=\Delta(v)-ev_0^2+e(fv_1)^2\geq 0.
$$
For instance, this would follow if we impose the constraint $\Delta_e(v):=\Delta(v)- ev_0^2\geq 0$; this assumption would  ensure the existence of Gieseker chambers before and after the autoequivalence $\Phi[1]$. Notice that $\Delta_e$ is invariant under tensoring by line bundles and dualizing, and also it is $\Phi[1]$-invariant.
\end{rem}

One application of our solution to equation \eqref{eq:mainZaBwT} has to do with projectivity of  Bridgeland moduli spaces. The precise result is the following.

\begin{thm}\label{projectiveModuliSpace}
Let $v=(v_0,v_1,v_2)$ be a Chern character with $v_0>0$, $fv_1> 0$, and $\Delta_e(v)\geq 0$. Let
$$
\bar{B}=\frac{\mu_f(v)}{2}\Theta+\lambda f,\ \ \bar{U}=\left(\frac{\mu_f(v)}{2}\right)^2\bar{V}
$$
where $\lambda \in \mathbb{Q}_{>0}$.  Then the moduli space $\mathcal{M}_{\bar{a},\bar{B},\bar{\omega}}(v)$ is projective for $\bar{\omega}$ integral with volume $\bar{V}\gg 0$ (or equivalently $\bar{b}\gg 0$) and $\lambda\gg 0$.
\end{thm}

\begin{rem}\label{FriedmanChamber}
Let us  fix a Chern character $v=(\ch_0,\ch_1,\ch_2)$ where $\ch_0$ and $f\ch_1$ are coprime.  Then  an ample divisor of the form $\omega = \Theta  + bf$ is  suitable in the sense of \cite[Definition 2.2]{FriedmanSO3} for $b \gg 0$ (see the proof of \cite[Lemma 2.3]{FriedmanSO3}, and also  \cite[Lemma 2.3]{FriedmanRank2}).  Given such an ample divisor $\omega$, the corresponding notions of  $\mu_\omega$-stability and $\omega$-Gieseker semistability coincide \cite[Section 1.1 and Proposition 7.1]{FMTes}.  Furthermore, the moduli space $M_\omega (v)$ of $\omega$-Gieseker semistable sheaves of Chern character $v$ is independent of $b$ for $b \gg 0$ (see \cite[Corollary 3.4]{FriedmanSO3} and \cite[Lemma 1.2]{yoshioka_1999}). We refer to these polarizations as polarizations in the \emph{Friedman chamber} and denote the corresponding moduli space by $M_f(v)$. In particular, under the assumption that $g.c.d.(v_0,fv_1)=1$, the moduli space $\mathcal{M}_{\bar{a},\bar{B},\bar{\omega}}(v)$ in Theorem \ref{projectiveModuliSpace} is isomorphic to the Friedman moduli space $M_f(\Phi(v)^{\vee})$.
\end{rem}

\begin{rem}
The point of Theorem \ref{projectiveModuliSpace} can now be stated as follows: By moving the ample class $\bar{\omega} = \Theta + \bar{b}f$ in the definition of a Bridgeland stability condition into the Friedman chamber (along with choices of $\bar{a}, \bar{B}$), we can show that the Bridgeland moduli space is projective.  This is in contrast to moving $\bar{\omega}$ along a ray into the Giesker chamber, which is a common way of obtaining a projective Bridgeland moduli space.
\end{rem}

\begin{proof}[Proof of Theorem \ref{projectiveModuliSpace}]
Our strategy is to apply  Lemma \ref{estimates_0fiberdegree} to   the Bridgeland stability condition $\sigma_{a, -B, \omega}$ and the Chern character $\Phi (v)^\vee$.  First, note that 
\begin{align*}
    \ch_0 (\Phi (v)^\vee) &= \ch_0 (\Phi (v)) = fv_1 > 0 \\
    f\ch_1(\Phi (v)^\vee) &= -f\ch_1 (\Phi (v)) = v_0 > 0\\
    \Delta (\Phi (v)^\vee) &= \Delta (\Phi (v)) = \Delta_e (v) + e(fv_1)^2 > 0 \\
    \mu_f (\Phi (v)^\vee) &= \frac{f\ch_1 (\Phi (v)^\vee)}{\ch_0 (\Phi (v)^\vee)} = \frac{v_0}{fv_1} = \frac{1}{\mu_f(v)}.
\end{align*}
On the other hand, under our assumptions on $\bar{V}$, $\bar{B}$, and $\bar{U}$, the solutions to \eqref{eq:V}, \eqref{eq:p}, \eqref{eq:q}, and \eqref{eq:U} become
\begin{align}
    V&=\frac{\mu_f(v)^2}{2}\bar{V}\label{especial:V}\\
    p&=-\frac{1}{\mu_f(v)}\label{especial:p}\\
    q&= \lambda + \frac{e}{2}\left( -\frac{1}{\mu_f(v)} - \frac{\mu_f(v)}{2} - 1 \right) \\
    U&=\frac{\bar{V}}{2}.\label{especial:U}
\end{align}
We now have 
\[
  -B = -(p\Theta + qf) = \mu_f( \Phi (v)^\vee) \Theta - qf
\]
as well as 
\begin{align*}
    f\ch_1^{-B}(\Phi(v)^{\vee})&=f(-\ch_1(\Phi(v))+B\ch_0(\Phi(v)))\\
    &=f\ch_1(\Phi(v)[1])-p\ch_0(\Phi(v)[1])\\
    &=\ch_0(v)+pf\ch_1(v)\\
    &=0
\end{align*}
and so
\begin{align*}
    \omega \ch_1^{-B}(\Phi (v)^\vee) &= \Theta \ch_1^{-B}(\Phi (v)^\vee) \\
    &= \Theta\ch_1(\Phi(v)^{\vee})+B\Theta\ch_0(\Phi(v)^{\vee})\\
    &=\lambda f\ch_1(v)-\ch_2(v)-\frac{e}{2}\left(\ch_0(v)+\frac{(f\ch_1(v))^2}{2\ch_0(v)}\right).
\end{align*}
Thus, $\omega \ch_1^{-B}(\Phi (v)^\vee)>0$ for $\lambda\gg 0$ (only depending on $v$), and Lemma \ref{estimates_0fiberdegree} says that the stability condition $\sigma_{a,-B,\omega}$ will lie in the Gieseker chamber for the class $\Phi(v)^{\vee}$ as long as $a$ satisfies the inequality of Lemma \ref{estimates_0fiberdegree} and when $\lambda \gg 0$.  (As explained in Remark \ref{DiscriminantBaA}, the conditions on the Chern character $v$ are given to guarantee the existence of a Gieseker chamber after the autoequivalence $\Phi[1]^D$.)  

 Next, we show that we can choose an integer $n$ so that $n(\omega \ch_1 (C))\omega^2 \in\mathbb{Z}$ for all $C \in D^b(X)$, as needed in the proof of Lemma \ref{estimates_0fiberdegree}. Notice that since $\bar{V}$ is integral,  from \eqref{especial:V} we know  $2\ch_0(v)^2V$ is also integral, but 
\[
  2\ch_0(v)^2V = 4\ch_0(v)^2 b - 2e\ch_0(v)^2
\]
so $4\ch_0(v)^2b$ is also integral, i.e.\ $b \in \frac{1}{4\ch_0(v)^2}\mathbb{Z}$. Therefore, we can  choose an integer $n$ (depending only on $\ch_0(v)$) such that $n(\omega \ch_1 (C))\omega^2$ is an integer for all $C \in D^b(X)$.

Applying Lemma \ref{estimates_0fiberdegree} to $\sigma_{a,-B,\omega}$ and $\Phi (v)^\vee$ for the above choice of $n$, we obtain that for 
\small{
\begin{equation}\label{eq:temp1}
\frac{U-e}{2}>\max\left\{\frac{n(\Theta\ch_1^{-B}(\Phi(v)^{\vee}))^3}{2\omega^2 f\ch_1(v)},\frac{n(\Theta\ch_1^{-B}(\Phi(v)^{\vee}))^3}{2\omega^2 f\ch_1(v)}-2n\ch_2^{-B}(\Phi(v)^{\vee})\Theta\ch_1^{-B}(\Phi(v)^{\vee})\right\}
\end{equation}
}
\normalsize
the moduli space $\mathcal{M}_{a,-B,\omega}(\Phi(v)^\vee)$ coincides with the Gieseker moduli space $M_{-B,\omega}(\Phi(v)^{\vee})$.  Since $\omega^2 = V = \frac{\mu_f(v)^2\bar{V}}{2}$, if we assume that $\bar{\omega}$ is integral (so $\bar{V}$ is an integer), then  $2\omega^2/\mu_f(v)^2 = \bar{V} \geq 1$, and so 
\[
  \frac{1}{\omega^2} \leq \frac{2\ch_0(v)^2}{(f\ch_1(v))^2}.
\]
As a result, if we require the  condition 
\small{
\begin{equation}\label{eq:temp2}
U-e>\max\left\{\frac{2n(\Theta\ch_1^{-B}(\Phi(v)^{\vee}))^3\ch_0(v)^2}{(f\ch_1(v))^3},\frac{2n(\Theta\ch_1^{-B}(\Phi(v)^{\vee}))^3\ch_0(v)^2}{(f\ch_1(v))^3}-2n\ch_2^{-B}(\Phi(v)^{\vee})\Theta\ch_1^{-B}(\Phi(v)^{\vee})\right\}
\end{equation}
}
\normalsize 
which is stronger than \eqref{eq:temp1}, then the moduli space $\mathcal{M}_{a,-B,\omega}(\Phi(v)[1]^{D})$ coincides with the Gieseker moduli space $M_{-B,\omega}(\Phi(v)^{\vee})$. This bound can be attained by taking $\bar{V}$ large enough by virtue of \eqref{especial:U}.

Finally, by taking $\bar{V}$ large enough it is also possible to attain the bound of Remark \ref{U-estimate-geom}.  Then together with Corollary \ref{cor:stabeq-1}, we obtain the isomorphisms
\[
\mathcal{M}_{\bar{a},\bar{B},\bar{\omega}}(v)\cong \mathcal{M}_{a,-B,\omega}(\Phi(v)[1]^{D})\cong M_{-B,\omega}(\Phi(v)^{\vee}).
\]
\end{proof}

The following theorem gives a possible relation between the Friedman moduli spaces before and after the relative Fourier-Mukai transform.

\begin{thm}\label{newisobyPhi}
Let $v=(v_0, v_1, v_2)$ be a Chern character vector with $v_0>0$, $fv_1>0$, $g.c.d(v_0,fv_1)=1$, and satisfying the condition on the discriminant $\Delta_e(v)+e(fv_1)^2\geq 0$. Let $E$ be a sheaf with $\ch(E)=v$ and $\bar{B}=\lambda f$. Suppose that there exists a constant $b_0$ such that $E$ is $\bar{B}$-twisted Gieseker semistable with respect to $\bar{\omega}=\Theta+\bar{b}f$ for all $\bar{b}>b_0$, and that there is a constant $\bar{a}_0$, only depending on $E$ and $\bar{B}$, such that for any $\bar{b}>b_0$ and any $\bar{a}>\bar{a}_0$, $E$ is $\sigma_{\bar{a},\bar{B},\bar{\omega}}$-semistable.  Then
    $$
    \Phi(E)[1]=F^D
    $$
for some sheaf $F$ that is Gieseker semistable with respect to $\Theta+bf$ for every $b\gg 0$. Moreover, if the moduli space of $\sigma_{\bar{a}, \bar{B}, \bar{\omega}}$-semistable objects is constant for $\bar{a}>\bar{a}_0$ and $\bar{b}>b_0$, then  $\Phi[1]^D$ induces an isomorphism between the Friedman moduli spaces $M_f(v)$ and $M_f(\Phi(v)^{\vee})$.
\end{thm}

The difference between the first and the second claim in Theorem \ref{newisobyPhi} is, that a priori, the bound $\bar{a}_0$ depends on the particular sheaf $E$.

\begin{proof}
As before, by Remark \ref{DiscriminantBaA}, the discriminant condition ensures that the Gieseker chambers of Corollary \ref{GiesekerChamber} exist for both the Chern characters $v$ and $\Phi (v)[1]$. Also, by Remark \ref{FriedmanChamber} we know that the Friedman moduli spaces before and after the relative Fourier-Mukai transform exist since 
$$
g.c.d(\ch_0(\Phi(v)^{\vee}),f\ch_1(\Phi(v)^{\vee}))=g.c.d(fv_1,v_0)=1.
$$
Now, the solutions \eqref{eq:V}-\eqref{eq:Ut1} are given by
\begin{equation}\label{sol_pbar=2}
V=\bar{U},\ U=\bar{V},\ B=\left(\lambda-\frac{e}{2}\right)f.
\end{equation}
Then by our hypothesis we can choose:
\begin{enumerate}
    \item $\bar{a}\in\mathbb{Z}$ so that 
    $\displaystyle \bar{a}>\max\left\{1/2e,\ \bar{a}_0\right\}$. By Proposition \ref{prop:anaughtbound}, this will guarantee that $\Phi[1]\cdot \sigma_{\bar{a},\bar{B},\bar{\omega}}$ is geometric. Moreover, we can take $\bar{a}$ large enough so that $\omega$ is in the Friedman chamber for $\Phi(v)^{\vee}$.
\item $\bar{b}\in\mathbb{Z}$ with $\bar{b}>b_0$ such that $\bar{\omega}$ is in the Friedman chamber for $v$.
\end{enumerate}
Now, for any $C \in D^b(X)$ notice that
\begin{align*}
2(\omega\ch_1(C))V&=2(\Theta\ch_1(C)+bf\ch_1(C))V\ \in\mathbb{Z},
\end{align*}
because $V=\bar{U}\in \mathbb{Z}$ and  $2b=V+e\in \mathbb{Z}$.

We will now apply Lemma \ref{estimates a-mini-walls} to $\omega, -B$ and $\Phi(v)^\vee$ with $n=2$.  Note that $\ch_0(\Phi(v)^\vee)=fv_1>0$, while
\begin{align*}
    \omega \ch_1^{-B}(\Phi(v)^\vee) &= \omega (\ch_1(\Phi(v)^\vee) + B\ch_0(\Phi(v)^\vee)) \\
    &= \Theta \ch_1(\Phi(v)^\vee) + bf\ch_1(\Phi(v)^\vee) + \left( \lambda -\frac{e}{2}\right)\ch_0(\Phi(v)^\vee)\\
    &=  \Theta \ch_1(\Phi(v)^\vee) + bv_0 + \left( \lambda -\frac{e}{2}\right) (fv_1).
\end{align*}
Therefore, by increasing $\bar{a}_0$ if necessary and noting $b=\bar{a}+e>\bar{a}_0+e$, we can assume $\omega \ch_1^{-B}(\Phi (v)^\vee)>0$.  Now, if we can choose $\bar{b}$ (which equals $a+e$) large enough so that 
\begin{equation}\label{condition3}
2a>2\mu_{-B,\omega}(\Phi(v)^{\vee})\Delta_{-B,\omega}(\Phi(v)^{\vee}),
\end{equation}
then  Lemma  \ref{estimates a-mini-walls} and (1) together would imply $\Phi(E)[1]^{D}\in M_f(\Phi(v)^{\vee})$ (recall our convention that  $(\_)^D=R\mathcal{H}om(\_,\mathcal{O})[1]$).

Writing $w=\Phi(v)[1]$, we have
\begin{align*}
\mu_{-B,\omega}(w^D)&=\frac{\omega \ch_1^{B}(w)}{-\ch_0(w)}\\
&= \frac{\Theta\ch_1(\Phi(v)[1]) +bf\ch_1(\Phi(v)[1])-\ch_0(\Phi(v)[1])(\lambda-e/2)}{fv_1}\\
&=\frac{-v_2+(b-e)v_0}{fv_1}+\lambda\\
&=\frac{-v_2+(\bar{U}-e)v_0/2}{fv_1}+\lambda,\ \text{ since $2b-e=V=\bar{U}$},\ \text{and}\\
\Delta_{-B,\omega}(w^D)&=(\omega\ch_1^B(w))^2-2\left(\ch_2(w)-\left(\lambda-\frac{e}{2}\right)f\ch_1(w)\right)\ch_0(w)V
\\&=\left(-v_2+(\bar{U}-e)\frac{v_0}{2}+\lambda fv_1\right)^2+2(\Theta v_1+efv_1-\lambda v_0)(fv_1)\bar{U}.
\end{align*}
Thus the right hand side of condition \eqref{condition3} only depends on $\bar{U}$ and $\bar{B}$. Therefore, in order to achieve \eqref{condition3}, we can fix $\bar{U}$ and take $\bar{V}$ (which equals $2\bar{b}-e$) to be large enough  so that 
$$
\bar{V}=U>2\mu_{-B,\omega}(w^D)\Delta_{-B,\omega}(w^D)+e
$$
which is equivalent to  \eqref{condition3}.  This completes the proof of the first claim.

For the second claim, note that under the extra hypothesis, 
 for $\bar{a}\gg 0$ the moduli space of $\sigma_{\bar{a},\bar{B},\bar{\omega}}$-semistable objects of Chern character $v$ is precisely the Friedman moduli space $M_f(v)$. 
\end{proof}

\begin{cor}\label{linebundles}
Let $X$ be a Weiersta\ss\ elliptic surface with Picard number $\rho(X)=2$ and $L$ be a line bundle with $L\cdot f>0$. Then there exists $b_0>0$ such that $\Phi(L)[1]^{D}$ is Gieseker stable with respect to $\omega=\Theta+bf$ for all $b> b_0$. 
\end{cor}
\begin{proof}
First, the condition $L\cdot f>0$ ensures that $\bar{\omega}\ch_1(L)>0$ for $\bar{b}\gg 0$. Also,
$$
\Delta_e(L)+e(L\cdot f)^2=e((L\cdot f)^2-1)\geq 0,
$$
and therefore there is a Gieseker chamber for the Chern character $\ch(\Phi(L)[1])$.

The conclusion will follow from Theorem \ref{newisobyPhi} once we prove that there is a bound independent of $\bar{b}$ for the walls destabilizing $L$ as long as $\bar{b}$ is large enough.

Thanks to the results of \cite{ArcaraMilesLB} on stability for line bundles on surfaces with a unique curve of negative self intersection, we know that if $\rho(X)=2$ then the only object possibly destabilizing $L$ is $L\otimes\mathcal{O}(-\Theta)$. 

Observe that
\[
(L-\Theta)\cdot\bar{\omega}=L\Theta+\bar{b}(Lf-1)+e
\]
and so we either have  $(L-\Theta)\cdot\bar{\omega}>0$ for $\bar{b} \gg 0$ or $(L-\Theta)\cdot\bar{\omega}\leq 0$ for $\bar{b} \gg 0$.  In the latter case, the sheaf $L(-\Theta)$ does not lie in the category $\Coh^{0,\bar{\omega}}(X)$ and so has no chance of destabilizing $L$, meaning $L$ is always  $\sigma_{\bar{a},0,\bar{\omega}}$-stable.

If, indeed, $L\otimes\mathcal{O}(-\Theta)$ produces a wall, then the value of $\bar{a}$ at which $L\otimes\mathcal{O}(-\Theta)$ $\sigma_{\bar{a},0,\bar{\omega}}$-destabilizes $L$ can be easily computed to be
$$
\bar{a}=\frac{L^2}{2}-\frac{L\Theta+\bar{b}Lf}{\bar{b}-e}\left(L\Theta+\frac{e}{2}\right).
$$
Now, notice that $L\Theta+\bar{b}Lf>0$ for $\bar{b}>-L\Theta/Lf$, and since
$$
\lim_{\bar{b}\rightarrow\infty}-\frac{L\Theta+\bar{b}Lf}{\bar{b}-e}\left(L\Theta+\frac{e}{2}\right)=-Lf\left(L\Theta+\frac{e}{2}\right),
$$
then there exists $\bar{b}_0>-L\Theta/Lf$ such that for all $\bar{b}>\bar{b}_0$
$$
-\frac{L\Theta+\bar{b}Lf}{\bar{b}-e}\left(L\Theta+\frac{e}{2}\right)<-Lf\left(L\Theta+\frac{e}{2}\right)+1.
$$
Thus, for all $\bar{b}>\bar{b}_0$ we have
$$
\bar{a}<\max\left\{\frac{L^2}{2},\frac{L^2}{2}-Lf\left(L\Theta+\frac{e}{2}\right)+1\right\}.
$$
\end{proof}
\begin{rem}
The results in Theorem \ref{newisobyPhi} are consistent with the results obtained in \cite[Theorem 1.1]{LLM}, where transforms of 1-dimensional sheaves were also studied. Moreover, we have the following analogous result.
\end{rem}
\begin{thm}\label{zerofiberdegree}
Let $v$ be a Chern character vector with $v_0>0$, $fv_1=0$ and $\Delta(v)\geq 0$, and let $\bar{B}=\lambda f$ be a $\mathbb{Q}$-divisor such that $\ch^{\bar{B}}_1(v)\cdot \Theta>0$. Then the set of $\bar{B}$-twisted Gieseker semistable sheaves of Chern character $v$ with respect to $\bar{\omega}=\Theta+\bar{b}f$ is independent of $\bar{b}$ for every integer $\bar{b}\gg 0$. Moreover, there exists $b_0>0$ depending only on $v$ and $\bar{B}$ such that  for every integer $\bar{b}\gg0$ the moduli space $M_{\bar{B},\Theta+\bar{b}f}(v)$ is isomorphic to a moduli space of 1-dimensional (twisted) Gieseker semistable sheaves with respect to $\Theta+b_0f$. 
\end{thm}

The main  idea behind the proof of Theorem \ref{zerofiberdegree} is as follows:   under the constraints  \eqref{eq:V}-\eqref{eq:Ut1}, the moduli of $\sigma_{a, B, \omega}$-semistable objects is constant for fixed $b$ and $a \gg 0$, while the isomorphic moduli space of $\sigma_{\bar{a}, \bar{B}, \bar{\omega}}$-semistable objects is identified with the moduli of twisted Gieseker semistable sheaves  when $\bar{b}$ (which corresponds to $a$) is integral.

\begin{proof}
This follows the same general steps as in the proof of Theorem \ref{newisobyPhi} with a few modifications.

 Let $E$ be a $\bar{B}$-twisted Gieseker semistable sheaf with respect to $\bar{\omega}$ of Chern character $v$, then by Lemma \ref{estimates_0fiberdegree}, since $\bar{\omega}=\Theta+\bar{b}f$ is integral, we know that for 
$$
\bar{a}>\max\left\{\frac{(\Theta\ch_1^{\bar{B}}(v))^3}{2\ch_0(v)\bar{\omega}^2}, \frac{(\Theta\ch_1^{\bar{B}}(v))^3}{2\ch_0(v)\bar{\omega}^2}-\ch_2^{\bar{B}}(v)\Theta\ch_1^{\bar{B}}(v)\right\},
$$
the sheaf $E$ is also $\sigma_{\bar{a},\bar{B},\bar{\omega}}$-semistable. Now, if we choose $\bar{b}_1>0$ so that $2(\Theta+\bar{b}_1f)^2>1$, then for every $\bar{b}\geq \bar{b}_1$ we have $1/(2\bar{\omega}^2)<1$, and so  there are no Bridgeland walls for $\bar{a}\geq a_0$ for every integer $\bar{b}\geq \bar{b}_1$, where
$$
a_0= \max\left\{\frac{(\Theta\ch^{\bar{B}}_1(v))^3}{\ch_0(v)},\ \frac{(\Theta\ch^{\bar{B}}_1(v))^3}{\ch_0(v)}-\ch_2^{\bar{B}}(v)(\Theta\ch^{\bar{B}}_1(v))\right\}.
$$
Next we fix $\bar{U}=2\bar{a}+e$ where $\bar{a}=\max\{a_0,1/2e\}$. This choice guarantees that the stability condition $\Phi[1]\cdot \sigma_{\bar{a},\bar{B},\bar{\omega}}$ is geometric and so it belongs to the same $\widetilde{\mathrm{GL}}^+(2,\mathbb{R})$-orbit as $\sigma_{a,B,\omega}$, where $a$, $B$, and $\omega$ are given by the solutions  \eqref{eq:V}-\eqref{eq:Ut1}. In volume and volume-like coordinates we have
\begin{equation}\label{sol_Bbar=f}
V=\bar{U},\ U=\bar{V},\ B=\left(\lambda-\frac{e}{2}\right)f.
\end{equation} 
This already proves the first part of the statement as for fixed $V$ and $U\gg 0$ (or equivalently, for fixed $\bar{a}$ and $\bar{b} \gg 0$) the moduli space of $\sigma_{a,B,\omega}$-semistable objects is constant and coincides with the moduli space of $B$-twisted Gieseker semistable sheaves of Chern character $w=\ch(\Phi(v)[1])$ with respect to $\omega$. To find a precise bound for $U=\bar{V}$ and so for $\bar{b}$, notice that
\begin{align*}
    \omega\ch_1(w) &= -\ch_2(v)-\frac{e}{2}\ch_0(v)+\frac{V}{2}\ch_0(v)=-\ch_2(v)-\frac{e}{2}\ch_0(v)+\frac{\bar{U}}{2}\ch_0(v),\ \text{and}\\
    \ch^{B}_2(w)&=\Theta\ch_1(v)-\lambda\ch_0(v),
\end{align*}
and so Lemma \ref{one dim real estimates} implies that for
$$
\bar{b}>\max\left\{\frac{\left(-\ch_2(v)-\frac{e}{2}\ch_0(v)+\frac{\bar{U}}{2}\ch_0(v)\right)^2}{2\bar{U}}+e,\frac{\left(-\ch_2(v)-\frac{e}{2}\ch_0(v)+\frac{\bar{U}}{2}\ch_0(v)\right)^2}{2\bar{U}}-\Theta\ch_1(v)+\lambda\ch_0(v)+e\right\}
$$
the Bridgeland moduli space $\mathcal{M}_{a,B,\omega}(w)$ coincides with $M_{B,\omega}(w)$.
\end{proof}

\section{The Euclidean algorithm}

We start this section with the first consequence of the well-behaved Bridgeland wall-crossing.
\begin{prop}\label{inducedrationalmap}
Let $v=(v_0,v_1,v_2)$ be a Chern character with $v_0>0$, $fv_1>0$ and $\Delta_e(v)\geq 0$. Let $\bar{B}=\bar{p}\Theta+\bar{q}f$ be a rational class such that $\bar{\omega}\ch_1^{\bar{B}}(v)>0$ for $\bar{\omega}=\Theta+\bar{b}f$. If $a,B$ and $\omega$ are as in  \eqref{usual:coordinates}, then the autoequivalence $\Phi[1]^{D}$ induces a rational map $f_{\Phi[1]^D}\colon M_{\bar{B},\bar{\omega}}(v)\dashrightarrow M_{-B,\omega}(\Phi(v)[1]^D)$, 
which for $\bar{a}\gg 0$ can be factored as
\small{
$$
\xymatrix{
M_{\bar{B},\bar{\omega}}(v) \ar_{\cong}[d]& \\
\mathcal{M}_{\bar{a},\bar{B},\bar{\omega}}(v) \ar_{\Phi[1]}^{\cong}[d]& \\ 
\mathcal{M}_{a,B,\omega}(\Phi(v)[1])\ar@{-->}[dd] \ar[dr] & \\
& \mathcal{M}_{a_{01},B,\omega}(\Phi(v)[1]) \\
\mathcal{M}_{a_1,B,\omega}(\Phi(v)[1])\ar[ur]\ar@{..}[dd] & \\
& \\
\mathcal{M}_{a_{k-1},B,\omega}(\Phi(v)[1])\ar@{-->}[dd]\ar[dr] & \\
& \mathcal{M}_{a_{0k},B,\omega}(\Phi(v)[1]) \\
\mathcal{M}_{a_k,B,\omega}(\Phi(v)[1])\ar[ur]\ar[d]_{(\_)^D}^{\cong} & \\
M_{-B,\omega}(\Phi(v)[1]^D) &
}
$$
}
\normalsize
where $a<a_{01}<a_1<a_{02}<a_2\cdots<a_{0k}<a_k$ and the values $a_{01},a_{02},\cdots,a_{0k}$ are precisely the walls for the Chern character $\Phi(v)[1]$ along the path $\{\gamma(t)=\sigma_{a+t,B,\omega}\}_{t>0}\subset \Stab^{\Gamma}(X)$. If the wall-crossing contractions are birational so is the map $f_{\Phi[1]^D}$, e.g., if $X$ is an elliptic K3 surface and there are no properly $\bar{B}$-twisted $\bar{\omega}$-semistable sheaves of Chern character $v$ \cite[Theorem 1.1]{BayerMacriMMP}.
\end{prop}
\begin{rem}\label{rem:chamberremark}
By Remark \ref{genericity},  the following statements are equivalent for a fixed Chern character $v$ and fixed $\bar{B}, \bar{\omega}$: 
\begin{itemize}
    \item There are no properly $\bar{B}$-twisted $\bar{\omega}$-semistable sheaves of Chern character $v$.
    \item Whenever the moduli of $\sigma_{\bar{a},\bar{B},\bar{\omega}}$-semistable objects of Chern character $v$ coincides with  the moduli of $\bar{B}$-twisted $\bar{\omega}$-Gieseker semistable sheaves of the same Chern character, the stability condition $\sigma_{\bar{a},\bar{B},\bar{\omega}}$ lies in a Bridgeland chamber for $v$. Moreover, if for some $\bar{a}_0$ the stability condition $\sigma_{\bar{a}_0,\bar{B},\bar{\omega}}$ lies on a Bridgeland chamber then $\sigma_{\bar{a},\bar{B},\bar{\omega}}$ lies on a Bridgeland chamber for all $\bar{a}\gg 0$.
\end{itemize}
The second statement is equivalent to saying that  $\sigma_{\bar{a},\bar{B},\bar{\omega}}$ is generic with respect to $v$ in the sense of \cite[Theorem 1.1]{BayerMacriMMP}   for $\bar{a}\gg 0$, which allows us to claim the birational wall-crossing maps.
\end{rem}

Now, let $E\in D^b(X)$ be an object with Chern character table given by
$$
v(E)=\begin{array}{|c|c|}\hline n & d\\ \hline c & s\\ \hline \end{array}
$$
where $n=\ch_0(E), d=f\ch_1(E), c=\Theta \ch_1(E)$ and $s=\ch_2(E)$.

Below, we define  two autoequivalences depending on $\ch(E)$.  We do not consider the case $n=d$ since we are only concerned with applications where $n, d$ are coprime:
\begin{itemize}
    \item If $n>d$ we can write $n=md+r$ for some integers $0\leq r<d$ and $m>0$. Consider the composition of autoequivalences $\Psi_{m}(E)= \Phi(\Phi(E)^{\vee}\otimes\mathcal{O}(-m\Theta))^{\vee}$, which numerically becomes:
\small{
$$
\begin{diagram}
\node{\begin{array}{|c|c|}\hline n & d\\ \hline c & s\\ \hline \end{array}}\arrow{s,l}{\Phi[1]}\arrow[2]{e,t}{\Psi_m}\node{}\node{\begin{array}{|c|c|}\hline r & d\\ \hline c-m(s+er)-\frac{e}{2}dm^2 & s\\ \hline \end{array}}\\
\node{\begin{array}{|c|c|}\hline -d & n\\ \hline -s+\frac{e}{2}d-en & c+ed-\frac{e}{2}n\\ \hline \end{array}}\arrow{s,l}{(\_)^D}\node{}\node{\begin{array}{|c|c|}\hline -r & d\\ \hline c-m(s+er)-\frac{e}{2}dm^2 & -s\\ \hline \end{array}}\arrow{n,r}{(\_)^D}\\
\node{\begin{array}{|c|c|}\hline d & n\\ \hline -s+\frac{e}{2}d-en & -c-ed+\frac{e}{2}n\\ \hline \end{array}}\arrow[2]{e,b}{\otimes\mathcal{O}(-m\Theta)}\node{}\node{\begin{array}{|c|c|}\hline d & r\\ \hline -s+\frac{e}{2}d-er & ms-c-ed\left(1-\frac{m^2}{2}\right)+er\left(\frac{1}{2}+m\right)\\ \hline \end{array}}\arrow{n,r}{\Phi[1]}
\end{diagram}
$$
}
\normalsize 
\item If $n<d$ we can write $d=kn+r$ for some integers $0\leq r<n$ and $k>0$. Consider the autoequivalence $\Upsilon_{k}(E)=E\otimes\mathcal{O}(-k\Theta)$:
\Large{
$$
\begin{diagram}
\node{\begin{array}{|c|c|}\hline n & d\\ \hline c & s\\ \hline \end{array}}\arrow{e,t}{\Upsilon_{k}}\node{\begin{array}{|c|c|}\hline n & r\\ \hline c +ekn& s-kc-\frac{e}{2}k^2n\\ \hline \end{array}}
\end{diagram}
$$
}
\normalsize
\end{itemize} 

\begin{thm}\label{ecuclidean}
Let $X$ be a Weierstra\ss\ elliptic K3 surface. Suppose that $v=(v_0,v_1,v_2)$ is a Chern character satisfying $v_0>0$, $fv_1>0$, $g.c.d.(v_0,fv_1)=1$ and $\Delta_e(v)\geq 0$. If $B=p\Theta+qf$, ${\omega}=\Theta+{b}f$, and there are no strictly $B$-twisted $\omega$-semistable sheaves of Chern character $v$, then there is a birational map
$$
M_{{B},{\omega}}(v)\dashrightarrow \Pic^0(X)\times\Hilb^{K}(X),
$$
for some $K\geq 0$.
\end{thm}

\begin{proof}
 First, we have a couple of observations about our hypotheses:
 \begin{itemize}
     \item As mentioned in Remark \ref{DiscriminantBaA}, $\Delta_e$ is invariant under tensoring by line bundles, dualizing, and also $\Phi[1]$-invariant. Thus, our assumption guarantees that $\Delta_e$ remains nonnegative after applying any of the autoequivalences $\Psi_m$ or $\Upsilon_k$ described above.
     \item Since $g.c.d.(v_0,fv_1)=1$ then the vector $v$ is primitive. The same remains true after applying any of the autoequivalences $\Psi_m$ or $\Upsilon_k$ described above.
 \end{itemize}

The division algorithm allows us to find nonnegative integers $m_0,\dots,m_t$, and $k_0,\dots,k_t$ such that if $\widetilde{\Psi}_v$ denotes the composition of autoequivalences
$$
\widetilde{\Psi}_v:=\Psi_{m_t}\circ\Upsilon_{k_t}\circ\cdots\circ\Psi_{m_1}\circ\Upsilon_{k_1}\circ\Psi_{m_0}\circ\Upsilon_{k_0},
$$
then either $\ch_0(\widetilde{\Psi}_v(v))=1$ or $f\ch_1(\widetilde{\Psi}_v(v))=k_{t+1}\ch_0(\widetilde{\Psi}_v(v))+1$ for some integer $k_{t+1}$. Consider the autoequivalence 
$$\Psi_v:=\begin{cases}\widetilde{\Psi}_v&\text{if}\ \ \ch_0(\widetilde{\Psi}_v(v))=1,\\ \Phi[1]^D\circ\Upsilon_{k_{t+1}}\circ\widetilde{\Psi}_v&\text{otherwise}.\end{cases}
$$
Then since the autoequivalences $\Upsilon_k$ induce isomorphisms between the twisted Gieseker moduli spaces
$$
g_k\colon M_{B',\omega'}(u)\xrightarrow[]{\cong} M_{B'-k\Theta,\omega'}(\Upsilon_k(u))
$$
for every $B'$, $\omega'$ and Chern character $u$ satisfying $\Delta(u)\geq 0$, then by Proposition \ref{inducedrationalmap} the autoequivalences
$\Psi_m$ induce birational maps
$$
f_m:=f_{\Phi[1]^D}\circ g_m\circ f_{\Phi[1]^D}\colon M_{B',\omega'}(u)\dashrightarrow M_{B'',\omega''}(\Psi_m(u))
$$
for some $B'', \omega''$ (see Remark \ref{rem:chamberremark}). Thus, the autoequivalence $\Psi_v$ induces a birational map between $M_{B,\omega}(v)$ and $M_{f}(1,\Lambda,\delta)$ for some $\Lambda\in \mathrm{NS}(X)$ and $\delta\in\frac{1}{2}\mathbb{Z}$. Now, since $\Delta(\Psi_v (E))\geq 0$, then 
$$
\delta=\frac{\Lambda^2}{2}-K
$$
for some integer $K\geq 0$. Thus, tensoring by $\mathcal{O}(-\Lambda)$ we obtain the desired birational map
$$
M_{B,\omega}(v)\dashrightarrow \Pic^0(X)\times \Hilb^K(X).
$$
\end{proof}


\begin{rem}
A version of the Euclidean algorithm of Theorem \ref{ecuclidean} was proven by Bernardara and Hein in \cite{bernardara:hal-00958218} for polarizations in the Friedman chamber, obtaining an isomorphism in that case. This should be thought as a limiting case of our result, specifically when  both $U$ and $V$ approach infinity. But even in the limiting case, the novelty of our approach is the decomposition of the birational map as a composition of Bridgeland wall-crossings. 
\end{rem}

\begin{rem}
In \cite[Corollary 7.7]{FMTes}, an isomorphism to a moduli space of 1-dimensional torsion-free sheaves was obtained on a component of the moduli space $M_f(v)$. On an arbitrary elliptic surface, the birationality of $M_f(v)$ with a moduli space of rank 1 torsion-free sheaves was first obtained by Bridgeland \cite[Theorem 1.1]{FMTes}, and by Yoshioka \cite[Theorem 1.14]{yoshioka_1999}, where irreducibility was also shown. In both cases, no assumption on the discriminant was required. 
\end{rem}

\begin{rem}
Even though only a birational morphism between $M_f(v)$ and $\mathrm{Pic}^0(X)\times \mathrm{Hilb}^K(X)$ was known, in \cite[Theorem 5.13]{Yosh2020} Yoshioka proved the equality of Euler numbers
$$
e(M_f(v))=e(\mathrm{Pic}^0(X)\times \mathrm{Hilb}^K(X)).
$$
Besides that, in the case where all multiple fibers have multiplicity 2, he estimates the codimension of the locus where the birational morphism is not an isomorphism (see \cite[Theorem 4.13]{Yosh2020}).
\end{rem}

\section{Elliptic K3 and Tramel-Xia's construction}
 The following lemma characterizes the effective curves that make the divisor $\Theta+ef$ strictly nef.

\begin{lem}\label{lem:TXprep1}
Let $p \colon X \to \mathbb{P}^1$ be a Weierstra{\ss} elliptic  surface and let $\Theta \subset X$ be any section where $\Theta^2=-e<0$. Then $\Theta + ef$ is a nef divisor and the only irreducible (effective) curve $W$ satisfying $W.(\Theta + ef)=0$ is $W=\Theta$.
\end{lem}

\begin{proof}
Take any irreducible curve $W \subset X$.  If $W\subseteq \Theta$, then $W=\Theta$ and $W.(\Theta + ef)=0$.  If $W \nsubseteq \Theta$, then we have two cases:
\begin{itemize}
\item $W$ is a vertical divisor, in which case $W.(\Theta + ef)>0$.
\item $W$ is not a vertical divisor, in which case the intersection of $W$ and $\Theta$ is transversal if nonempty.  That is, $W.\Theta \geq 0$ while $W.f>0$ \cite[Lemma 3.15]{Lo11}, giving  us  $W.(\Theta + ef) > 0$.
\end{itemize}
This completes the proof of the lemma.
\end{proof}

For the rest of the section,  we will  consider  Weierstra{\ss} elliptic fibrations of the form $p\colon X \to Y=\mathbb{P}^1$ where $X$ is a K3 surface. In particular, we have  $e=2$ since $K_X\equiv (2g_Y-2+e)f$. For surfaces that contain rational curves of negative self-intersection, Tramel-Xia constructed in \cite{tramel2017bridgeland} a family of stability conditions using nef divisors instead of the ample divisors $\omega$ from the previous sections. These stability conditions sit on the boundary of the geometry chamber, and we will study their behaviour under the action of  the relative Fourier-Mukai transform. 

Using Lemma \ref{lem:TXprep1}, we can now specialize  Tramel-Xia's main results to our setting, giving us:

\begin{thm}\label{TramelXiaSC} (\cite[Theorems 5.4 and 6.11, Lemma 7.2]{tramel2017bridgeland}) Let $p\colon X\rightarrow \mathbb{P}^1$ be a Weierstra\ss\ elliptic K3 surface with a section $\Theta$. Let $\bar{B}=\lambda\Theta$, $\bar{H}=\Theta+2f$, $k\in \mathbb{Z}$ be an integer such that $k+1<-2\lambda<k+2$, and $z\in\mathbb{R}$ such that $z>\lambda^2$. Then there is a heart $\mathcal{B}_{\bar{H},k}\subset D^b(X)$ of a bounded t-structure such that the pair $\sigma^k_{z,\bar{B},\bar{H}}=(Z_{z,\bar{B},H},\mathcal{B}_{\bar{H},k})$ is a stability condition, where
$$
Z_{z,\bar{B},\bar{H}}(E)=-\ch_2^{\bar{B}}(E)+(z-\lambda^2)\ch_0(E)+i\bar{H}\ch_1^{\bar{B}}(E).
$$
In addition, the skyscraper sheaf $\mathcal{O}_x$ is $\sigma^k_{z,\bar{B},\bar{H}}$-stable when $x \notin \Theta$, and $\sigma^k_{z,\bar{B},\bar{H}}$-strictly semistable if $x \in \Theta$ with Jordan-Hölder filtration 
$$
0\to \mathcal{O}_{\Theta}(k+1)\to \mathcal{O}_x\to \mathcal{O}_{\Theta}(k)[1]\to 0.
$$
\end{thm}

We quickly recall the construction of the heart $\mathcal{B}_{\bar{H},k}$ in \cite{tramel2017bridgeland}. Consider the Mumford slope of a sheaf with respect to the nef class $\bar{H}$:
$$
\mu_{\bar{H}}(E)=\begin{cases}\frac{\bar{H}\ch_1(E)}{\ch_0(E)}&\text{if}\ \ \ch_0(E)\neq 0,\\
+\infty&\text{otherwise}.\end{cases}
$$
Then, as usual, the full subcategories
\begin{align*}
 \mathcal{T}_{\bar{H}}&=\{E\in\Coh(X)\colon \mu_{\bar{H}}(Q)>0\ \text{for all quotient sheaves}\ E\twoheadrightarrow Q\}, \\
 \mathcal{F}_{\bar{H}}&=\{E\in\Coh(X)\colon \mu_{\bar{H}}(F)\leq 0\ \text{for all subsheaves}\ F\hookrightarrow E\}
\end{align*}
form a torsion pair in $\Coh(X)$, allowing us to  form the tilted heart $\Coh^{\bar{H}}(X)=\langle\mathcal{F}_{\bar{H}}[1],\mathcal{T}_{\bar{H}}\rangle$. Now, the full subcategories
\begin{align*}
 \mathcal{T}_{\bar{H},k}&=\{E\in\Coh^{\bar{H}}(X)\colon \Hom(E,\mathcal{O}_{\Theta}(i))=0\ \text{for all}\ i\leq k\}, \\
\mathcal{F}_{\bar{H},k}&=\langle \mathcal{O}_{\Theta}(i)\colon i\leq k\rangle
\end{align*}
form a torsion pair in $\Coh^{\bar{H}}(X)$ by \cite[Lemma 3.2]{tramel2017bridgeland}, allowing us to tilt  $\Coh^{\bar{H}}(X)$ to obtain the heart
$$
\mathcal{B}_{\bar{H},k}=\langle\mathcal{F}_{\bar{H},k}[1],\mathcal{T}_{\bar{H},k}\rangle.
$$
Notice that objects in $\mathcal{B}_{\bar{H},k}$ are still 2-term complexes sitting at degrees $-1$ and $0$.

\begin{prop}\label{prop:Oxtransfstable}
Let $\sigma^k_{z,\bar{B},\bar{H}} = (Z_{z,\bar{B},\bar{H}}, \mathcal{B}_{\bar{H},k})$ be as in Theorem \ref{TramelXiaSC}.  Then for any $x \in X$, the skyscraper sheaf $\OO_x$ is a $(\Phi \cdot \sigma^k_{z,\bar{B},\bar{H}})$-stable object.
\end{prop}

\begin{proof}
From the definition of the action of $\Aut (D^b(X))$ on $\mathrm{Stab}(X)$, and noting $\Phi^{-1} \cong \whPhi[1]$, it suffices to show that all the sheaves $\whPhi \OO_x$ are $\sigma^k_{z,\bar{B},\bar{H}}$-stable of the same phase as $x$ runs through all the closed points of $X$.

Fix any $x\in X$, we first check that $\whPhi \OO_x \in \mathcal{B}_{\bar{H},k}$.  To begin with, since $\whPhi \OO_x$ is a torsion sheaf (in fact, it is supported on the fiber of $p$ containing $x$), it lies in the torsion class $\Tc_{\bar{H}}$.  Hence $\whPhi \OO_x$ lies in the first tilt $\Coh^{\bar{H}}(X)$.  Next, every $\OO_{\Theta} (i)$ is a pure 1-dimensional sheaf supported on $\Theta$, which is a section, and so $\Hom (\whPhi \OO_x, \OO_{\Theta} (i))=0$ for every $i \in \mathbb{Z}$.  This implies $\whPhi \OO_x \in \Tc_{\bar{H},k} \subset \Bc_{\bar{H},k}$.

Now take any $\Bc_{\bar{H},k}$-short exact sequence
\[
0 \to M \to \whPhi \OO_x \to N \to 0
\]
where $M, N \neq 0$.  Since $\Im Z_{z,\bar{B},\bar{H}} = \bar{H}\ch_1$ is nonnegative on  $\Bc_{\bar{H},k}$, we have
$$
0\leq \bar{H}\ch_1(M)\leq \bar{H}\ch_1(\whPhi\mathcal{O}_x)=\bar{H}f=1.
$$
In the case where $\bar{H}\ch_1(M)=1$, we have $\phi_{Z_{z,\bar{B},\bar{H}}}(N)=1>\phi_{Z_{z,\bar{B},\bar{H}}}(\whPhi \OO_x)$, so let us assume $\bar{H}\ch_1(M)=0$ from now on.

Consider the $\Bc_{\bar{H},k}$-short exact sequence
\[
0 \to \Hc^{-1}_{0}(M)[1] \to M \to \Hc^0_{0}(M) \to 0,
\]
where we write $\mathcal{H}^{i}_0$ for the cohomologies of an object with respect to the t-structure generated by the heart $\Coh^{\bar{H}}(X)$.  Again, since $\Im Z_{z,\bar{B},\bar{H}} = \bar{H}\ch_1$ is nonnegative on  $\Bc_{\bar{H},k}$, it follows that $\bar{H}\ch_1(\mathcal{H}^{-1}_0(M))=0=\bar{H}\ch_1(\Hc^0_{0}(M))$.  Next, write $T =  \Hc^0_{0}(M)$ and consider the short exact sequence in $\Coh^{\bar{H}}(X)$
$$
0 \to \mathcal{H}^{-1}(T)[1] \to T \to \mathcal{H}^0(T)\to 0
$$
where $\mathcal{H}^{-1}(T), \mathcal{H}^0(T)$ are coherent sheaves.

Since $\bar{H}\ch_1$ is nonnegative on $\Coh^{\bar{H}}(X)$, that $\bar{H}\ch_1(T)=0$ implies $\bar{H}\ch_1(\mathcal{H}^0(T))=0$, which in turn implies that $\mathcal{H}^0(T)$ is a torsion sheaf supported on $\Theta$ (here we use Lemma \ref{lem:TXprep1}).  Also, since $\Theta$ is a section and $\whPhi \OO_x$ is a pure 1-dimensional fiber sheaf \cite[Lemma 6.22]{FMNT}, we have $\Hom (\mathcal{H}^0(T),\whPhi \OO_x)=0$.  Hence
\begin{align*}
 \Hom_{D^b(X)} (M,\whPhi \OO_x) &\cong \Hom_{D^b(X)} (\mathcal{H}^0_0(M), \whPhi \OO_x) \text{\quad since $\mathcal{H}^{-1}_0(M)[1] \in D^{\leq -1}_{\Coh (X)}(X)$} \\
 &= \Hom_{D^b(X)}(T,\whPhi \OO_x) \\
  &\cong \Hom_{D^b(X)} (\mathcal{H}^0(T),\whPhi \OO_x) \text{\quad since $\mathcal{H}^{-1}(T)[1] \in D^{\leq -1}_{\Coh (X)}(X)$}\\
  &= 0
\end{align*}
contradicting the assumption that $M$ is a nonzero subobject of $\whPhi \OO_x$ in the heart $\Bc_{\bar{H},k}$.  Hence $\whPhi \OO_x$ must be $\sigma^k_{z,\bar{B},\bar{H}}$-stable.
\end{proof}

\begin{cor}\label{TXafterFM}
With $k, z, \bar{B}, \bar{H}$ as in Theorem \ref{TramelXiaSC}, there exist $a\in \mathbb{R}, g\in \widetilde{\mathrm{GL}}^{+}\!(2,\mathbb{R})$, and $\mathbb{R}$-divisors of the form  $B=p\Theta+qf$, $\omega=\Theta+bf$ with $b>2$ such that the stability condition $\sigma_{a, B, \omega}$ satisfies
\begin{equation}\label{eq:TXtoGeometric}
\Phi\cdot \sigma^k_{z,\bar{B},\bar{H}}=\sigma_{a,B,\omega}\cdot g.
\end{equation}
\end{cor}
\begin{proof}
We know $\Phi\cdot \sigma^k_{z,\bar{B},\bar{H}}$ is a geometric stability condition from Proposition \ref{prop:Oxtransfstable}.  Therefore,  by Proposition \ref{isoUpsilon} it is enough to show that there is a solution to the central charge equation associated to \eqref{eq:TXtoGeometric}, i.e.\ there exists a matrix $T \in \mathrm{GL}^+\!(2,\mathbb{R})$ such that 
$$
Z_{a,B,\omega}(\Phi(E))=TZ_{z,\bar{B},\bar{H}}(E)  \text{\quad for all $E \in D^b(X)$}.
$$
Notice that the central charge $Z_{z,\bar{B},\bar{H}}$ can be written in the form $Z_{\bar{a}, \bar{B}, \bar{\omega}}$ with $\bar{a}=z-\lambda^2$, $\bar{B}=\lambda \Theta$, and  $\bar{\omega}=\bar{H}=\Theta + 2f$. Setting 
$$
    \bar{U}=2(z-\lambda^2)+e=2(z-\lambda^2+1)\ \ \text{and}\ \ \bar{V}=\bar{H}^2=2,
$$
the solutions \eqref{eq:V}, \eqref{eq:p}, \eqref{eq:q}, and \eqref{eq:U} read
\begin{align*}
    V &=2z+2\\
    p&=-\frac{\lambda}{z+1}\\
    q&=-\lambda\left(\frac{1}{z+1}+1\right)-1\\
    U&=2\left(1-\frac{\lambda^2}{z+1}\right).
\end{align*}
Thus we obtain
$$
\omega=\Theta+(z+2)f,\ \ a=-\frac{\lambda^2}{z+1},\ \text{and}\ B=-\frac{\lambda}{z+1}(\Theta+f)-(\lambda+1)f,
$$
and from \eqref{usual:coordinates}
$$
T=\begin{pmatrix}\frac{\lambda}{z+1} & 1-\frac{\lambda^2}{z+1}\\ -1 & \lambda \end{pmatrix}.
$$
Finally, notice that $\omega$ is ample since  $z>0$ by assumption.
\end{proof}

\begin{rem}
Notice that the geometric stability condition $\sigma_{a,B,\omega}$ obtained in Corollary \ref{TXafterFM} satisfies
$$
-1+\frac{1}{z+1}<a<0,
$$
which is very special for our type of K3 surfaces and deviates from the usual assumption of $a>0$ in Bridgeland stability conditions. In fact, the positivity property of the central charge $Z_{a,B,\omega}$ on the heart $\mathrm{Coh}^{B,\omega}$ does not follow from the usual Bogomolov inequality for $a<0$.  On a K3 surface, however, we do have a stronger version of the Bogomolov-Gieseker inequaity, which says that for a slope stable torsion-free sheaf $A$, we have \cite[Section 6]{ABL}
\[
  \ch_2(A) \leq \frac{\ch_1(A)^2}{2\ch_0(A)} - \ch_0(A) + \frac{1}{\ch_0(A)}
\]
or equivalently
\[
  \Delta (A) \geq 2\ch_0(A)^2-2.
\]
Since $\Delta (A)$ is invariant under twising the Chern character by a 
$B$-field, we also have
\begin{equation}\label{eq:strongBGtwist}
 \ch_2^B(A) \leq \frac{\ch_1^B(A)^2}{2\ch_0(A)} - \ch_0(A) + \frac{1}{\ch_0(A)}.
\end{equation}
Under certain conditions on $B$ and $\omega$, we can use this stronger Bogomolov-Gieseker inequality to produce Bridgeland stability conditions $\sigma_{a, B, \omega}$ with $a<0$ as in the lemma below.

\end{rem}
\begin{lem}\label{CCwithnegativea}
Let $X$ be a K3 surface and $B,\omega\in\mathrm{NS}(X)\otimes\mathbb{Q}$ with $\omega$ integral and ample, and such that $B\omega\notin \mathbb{Z}$. Then 
$$
Z_{a,B,\omega}=-\ch_2^B+a\ch_0+i\omega\ch_1^B
$$
satisfies the positivity property for a central charge on $\Coh^{B,\omega}(X)$ for every $a>-3/4$.  If we also have $a \in \mathbb{Q}$ and $B \in \mathrm{NS}(X)\otimes \mathbb{Q}$, then $Z_{a, B, \omega}$ is a stability function on $\Coh^{B,\omega}$ with the HN property.
\end{lem}

\begin{proof}
Following the standard argument, it suffices  to check that if $E$ is a slope stable torsion-free  sheaf with $\omega\ch_1^B(E)=0$, then $\Re Z_{a,B,\omega}(E[1])<0$.  Notice that we can assume $\ch_0(E)>1$, for otherwise the assumption $\omega \ch_1^B(E)=0$ gives
$$
\omega\ch_1(E)=B\omega
$$
where the left-hand side is an integer, contradicting our assumption that $B\omega$ is not an integer.  We now have 
\begin{align*}
    \Re Z_{a,B,\omega}(E[1]) &= -\ch_2^B(E[1]) + a\ch_0(E[1]) \\
    &= \ch_2^B(E) - a\ch_0(E) \\
    &\leq \frac{(\ch_1(E)^B)^2}{2\ch_0(E)} - \ch_0(E) + \frac{1}{\ch_0(E)} -a\ch_0(E) \text{\quad by \eqref{eq:strongBGtwist}} \\
    &\leq -\ch_0(E) + \frac{1}{\ch_0(E)} -a\ch_0(E) \text{\quad by the Hodge Index Theorem} \\
    &= -(a+1)\ch_0(E) +  \frac{1}{\ch_0(E)}.
\end{align*}
Now it is easy to see that when $a>-3/4$, we have 
\[
  -(a+1)\ch_0(E) + \frac{1}{\ch_0(E)} < 0
\]
whenever $\ch_0(E) \geq 2$, giving us $\Re Z_{a,B,\omega}(E[1]) < 0$.

To see the second claim, note that when $a, B$ are both over $\mathbb{Q}$, the usual Bridgeland stability condition $\sigma_{B,\omega}$ is algebraic in the sense of \cite[Definition 2.5]{TODA20082736}, and so the heart $\Coh^{B,\omega}$ is noetherian.  Since the image of $Z_{a, B, \omega}$ is now contained in $\mathbb{Q} \oplus \mathbb{Q}i$, its HN property as a stability function on $\Coh^{B,\omega}$ follows from \cite[Proposition 3.4]{LZ2}.
\end{proof}

\begin{rem}
The central charge $Z_{a,B,\omega}$ obtained in Corollary \ref{TXafterFM} satisfies the hypothesis of Lemma \ref{CCwithnegativea} when $z\in\mathbb{Z}$ and $2\lambda\notin \mathbb{Z}$ since
$$
B\omega=-2\lambda-1.
$$
This suggests that  for general values of $\lambda$ and $z$, the positivity of $Z_{a,B,\omega}$  in Corollary \ref{TXafterFM} is rather a consequence of our solution to \eqref{eq:TXtoGeometric} using the relative Fourier-Mukai transform.  This demonstrates the idea that autoequivalences can be used to establish the positivity of the central charge in a Bridgeland stability condition where available Bogomolov-Gieseker inequalities are not sufficient.
\end{rem}

\begin{cor}\label{projectivityTXcomponent}
Let $v$ be a fixed Chern character, and  suppose  $\sigma^k_{z,\bar{B},\bar{H}}$ is a stability condition from Theorem \ref{TramelXiaSC} that   is contained in a chamber for $v$. Then there is a projective coarse moduli space parametrizing families of S-equivalence classes of $\sigma^k_{z,\bar{B},\bar{H}}$-semistable objects of Chern character $v$.
\end{cor}
\begin{proof}
By \cite[Theorem 1.3]{bayer2014projectivity} we know that if $\sigma_{a,B,\omega}$ is generic with respect to the Chern character $\Phi(v)$, then the coarse moduli space $\mathcal{M}_{a,B,\omega}(\Phi(v))$ exists as a normal irreducible projective variety with $\mathbb{Q}$-factorial singularities, in which case the same would be true for the moduli space of $\sigma^k_{z,\bar{B},\bar{H}}$-semistable objects by Corollary \ref{TXafterFM}. The genericity of $\sigma_{a,B,\omega}$ with respect to $\Phi(v)$ is equivalent to the genericity of $\sigma^k_{z,\bar{B},\bar{H}}$ with respect to $v$, since the action of $\Phi$ sends short exact sequences to short exact sequences and preserves inequalities of phases.
\end{proof}


\appendix

\section{Computations}\label{computations}

Our goal is to find conditions on the parameters $a, B, \omega$ and $\bar{a}, \bar{B}, \bar{\omega}$, along with   a matrix $W$ such that the relation
\begin{equation}\label{eq:cenchareq}
  Z_{a,B,\omega}(\Phi E) = W Z_{\bar{a},\bar{B},\bar{\omega}}(E)
\end{equation}
holds for all $E \in D^b(X)$.  The computation can be slightly simplified if we set
\[
  Z_{a,B,\omega}' (E) = \begin{pmatrix} 1 & -fB \\ 0 & 1 \end{pmatrix} Z_{a,B,\omega}(E)
\]
for any $a, B, \omega$ and consider the equivalent relation
\begin{equation}\label{eq:rel1}
    Z_{a,B,\omega}'(\Phi E) = W' Z_{\bar{a},\bar{B},\bar{\omega}}'(E)
\end{equation}
for some $W'\in GL(2,\mathbb{R})$.  Recall our notation
\begin{gather*}
  \bar{\omega}= \Theta + \bar{b}f, \text{\quad} \bar{B} = \bar{p} \Theta + \bar{q} f, \\
  \omega = \Theta + bf, \text{\quad} B = p\Theta + qf.
\end{gather*}
From the computation in \cite[8.5]{Lo20} (where the notation $Z_{a,b,B}, Z_{a,b,B}'$ correspond to our notation $Z_{a,B,\omega}, Z_{a,B,\omega}'$, respectively), we know that the relation \eqref{eq:rel1} holds for all $E \in D^b(X)$ if and only if $W' = \begin{pmatrix} 0 & 1 \\ -1 & 0 \end{pmatrix}$ and the following equations are satisfied:
\begin{align}
 - \bar{b}\bar{p} +  \bar{q}  &= \tfrac{e}{2} +  (b-e)p + q  \label{eq:cons1} \\
  \bar{a} + (\bar{b} - \tfrac{e}{2})\bar{p}^2  &= b - e  \label{eq:cons2} \\
  \bar{b}-e &=  a + (b- \tfrac{e}{2})  p^2   \label{eq:cons3} \\
   (\bar{b}-e) \bar{p}+\bar{q}  &= \tfrac{e}{2}  - bp+ q \label{eq:cons4}.
\end{align}
Note the symmetry between \eqref{eq:cons2} and \eqref{eq:cons3}, and the almost-symmetry between \eqref{eq:cons1} and \eqref{eq:cons4}.  Also, with the above choice of $W'$ and under these constraints, we obtain a solution to \eqref{eq:cenchareq} with
\[
  W = \begin{pmatrix} 1 & fB \\ 0 & 1 \end{pmatrix} \begin{pmatrix} 0 & 1 \\ -1 & 0 \end{pmatrix} \begin{pmatrix} 1 & -f\bar{B} \\ 0 & 1 \end{pmatrix} = \begin{pmatrix} -p & 1+p\bar{p} \\ -1 & \bar{p} \end{pmatrix}.
\]

Now let us fix $\bar{a}, \bar{b}, \bar{p}, \bar{q}$ as input, and solve for $a, b, p, q$  in \eqref{eq:cons1}-\eqref{eq:cons4} as output.  Let us also assume  $\bar{a}>0$ and $\bar{b} \geq e$.  We can immediately solve for $b$ from \eqref{eq:cons2}:
\begin{equation}\label{eq:cons5}
  b = \bar{a} + (\bar{b} - \tfrac{e}{2})\bar{p}^2 + e.
\end{equation}
From this we see  $b>e$.  Next, \eqref{eq:cons1} and \eqref{eq:cons4} together can be written as the matrix equation
\[
  \begin{pmatrix} b-e & 1 \\ -b & 1 \end{pmatrix}\begin{pmatrix} p \\ q \end{pmatrix} = \begin{pmatrix}  - \bar{p}\bar{b} + \bar{q} - \tfrac{e}{2} \\
  \bar{p}(\bar{b} - e)  + \bar{q} -\tfrac{e}{2} \end{pmatrix}.
\]
where $\det{\begin{pmatrix} b-e & 1 \\ -b & 1 \end{pmatrix}} = 2b -e >0$. Hence we can solve for $p, q$ as 
\begin{align}
\begin{pmatrix} p \\ q \end{pmatrix} &=\begin{pmatrix} b-e & 1 \\ -b & 1 \end{pmatrix}^{-1} \begin{pmatrix}  - \bar{p}\bar{b} + \bar{q} - \tfrac{e}{2} \\
  \bar{p}(\bar{b} - e)  + \bar{q} -\tfrac{e}{2} \end{pmatrix} \notag \\
  &= \tfrac{1}{2b -e} \begin{pmatrix} 1 & -1 \\ b & b-e \end{pmatrix} \begin{pmatrix}  - \bar{p}\bar{b} + \bar{q} - \tfrac{e}{2} \\
  \bar{p}(\bar{b} - e)  + \bar{q} -\tfrac{e}{2} \end{pmatrix}\notag\\
  &= \tfrac{1}{2b -e} \begin{pmatrix} -\bar{p}(2\bar{b}-e) \\
  (2b-e)(\bar{q}-\tfrac{e}{2}) + e\bar{p}(-b-\bar{b} + e)  \label{eq:pq} \end{pmatrix}.
\end{align}
Lastly, from \eqref{eq:cons3} we have
\begin{equation}\label{eq:asol1}
 a = \bar{b} - e - (b-\tfrac{e}{2})p^2.
\end{equation}
We have now shown that   $a, b, p, q$ can all be solved for in terms of $\bar{a}, \bar{b}, \bar{p}, \bar{q}$.  That is, we can think of $\bar{a}, \bar{b}, \bar{p}, \bar{q}$ (the `barred' parameters) as inputs and $a, b, p, q$ (the `unbarred' parameters) as outputs when solving an equation of the form \eqref{eq:cenchareq}.  

\begin{rem}\label{rem:A1}
If $\bar{a},\bar{b}, \bar{p}, \bar{q} \in \mathbb{Q}$, then we also have $a, b, p, q \in \mathbb{Q}$; this is needed in Theorem  \ref{newisobyPhi}, for instance.
\end{rem}

Nonetheless, it can be useful to mix the barred and unbarred parameters to exploit the symmetry.  For instance, we have
\begin{equation}\label{eq:passoc2}
    p(2b-e)=-\bar{p}(2\bar{b}-e)
\end{equation}
and
\begin{align}
    q&=\bar{q}-\frac{e}{2}-e\bar{p}\frac{b-\frac{e}{2}}{2b-e}-e\bar{p}\frac{\bar{b}-\frac{e}{2}}{2b-e}\\
    &=\bar{q}+\frac{e}{2}(p-\bar{p}-1) \notag
\end{align}
where the last equality holds because we have $\bar{p}(\bar{b}-\tfrac{e}{2})=-\tfrac{1}{2}p(2b-e)$ from \eqref{eq:passoc2}.  We can also rewrite  \eqref{eq:asol1} as
\begin{align}
a &= \bar{b} - e - \frac{1}{2} \frac{(2\bar{b}-e)^2}{(2b-e)}\bar{p}^2   \text{\qquad by rewriting $p$ using  \eqref{eq:pq}} \notag\\
&= \bar{b}-e - \frac{(2\bar{b}-e)^2 \bar{p}^2}{2(2\bar{a}+(2\bar{b}-e)\bar{p}^2+e)} \text{\qquad by rewriting $b$ using \eqref{eq:cons5}}. \label{eq:cons6}
\end{align}

\begin{rem}
When $\bar{b}$ and $\bar{p}$ are fixed and as $\bar{a}\to \infty$, we see from \eqref{eq:cons6} that  $a \to (\bar{b}-e)$.  In particular, when $\bar{b}>e$, for $\bar{a}\gg 0$ we have $a>0$. In terms of the volume and volume-like coordinates introduced in Section \ref{preservation:gieseker}, this says that as $\bar{U}\rightarrow \infty$, we have  $U\rightarrow \bar{V}$.
\end{rem}

\bibliographystyle{hep}
\bibliography{refs-perpetual.bib}

\end{document}